\documentclass[12pt]{amsart}
\pdfoutput=1
\usepackage{setspace}
\usepackage[margin=1in,marginparwidth=0.8in, marginparsep=0.1in]{geometry}

\usepackage[bookmarks=true, bookmarksopen=true,%
bookmarksdepth=3,bookmarksopenlevel=2,%
colorlinks=true,%
linkcolor=blue,%
citecolor=blue,%
filecolor=blue,%
menucolor=blue,%
urlcolor=blue]{hyperref}

\usepackage[dvipsnames]{xcolor}
\usepackage{graphicx}

\usepackage{tikz, tikz-3dplot}
\usepackage{caption, subcaption}
\usepackage{amssymb}
\usepackage[all]{xy}
\usepackage{mathtools}
\usepackage{comment}

\newcommand{\bC}{\mathbb{C}}

\newcommand{\bN}{\mathbb{N}}

\newcommand{\bR}{\mathbb{R}}
\newcommand{\bZ}{\mathbb{Z}}

\newcommand{\cE}{\mathcal{E}}

\newcommand{\cH}{\mathcal{H}}

\newcommand{\cT}{\mathcal{T}}

\numberwithin{equation}{subsection}
\numberwithin{figure}{subsection}

	\newcommand{\e}{\mathrm{exp}}

	\newcommand{\x}{\mathbf{x}}
	\newcommand{\y}{\mathbf{y}}
	
	\newcommand{\Ad}{\mathrm{Ad}}

	\newcommand{\Sk}{\mathrm{Sk}}

	\newcommand{\skein}[1]{P_{#1}}

\newtheorem{dummy}{dummy}[section]
\newtheorem{lemma}[dummy]{Lemma}
\newtheorem{theorem}[dummy]{Theorem}
\newtheorem*{untheorem}{Main Theorem}

\newtheorem{corollary}[dummy]{Corollary}
\newtheorem{proposition}[dummy]{Proposition}

\theoremstyle{definition}
\newtheorem{definition}[dummy]{Definition}

\newtheorem{example}[dummy]{Example}
\newtheorem{remark}[dummy]{Remark}

\newtheorem{notation}[dummy]{Notation}

\title{A proof of the pentagon relation for skeins}

\author{Mingyuan Hu
\\
{\tiny Department of Mathematics, Northwestern University} }

\begin{document}
	
	\begin{abstract}
	In \cite{HSZ23}, with Gus Schrader and Eric Zaslow we developed a skein-theoretic version of cluster theory, and made a conjecture on the pentagon relation for the skein dilogarithm. Here we give a topological proof of this conjecture. Combining \cite{MS21} and \cite{BCMN23}, we get a surjection from the skein algebra $\Sk^+(T - D)$ to the positive part of the elliptic Hall algebra $\cE_{q, t}^+$. Hence our pentagon relation generalizes the ones in \cite{Z23} and \cite{GM19}. 
	\end{abstract}

    \maketitle
    
%	\setcounter{tocdepth}{1}
%	\tableofcontents
	
	\section{introduction}

The pentagon relation arises in many areas of mathematics and physics. 
For a family of operators $\{Q\}$, it takes the form:
\[
	Q \cdot Q = Q \cdot Q  \cdot Q
\]
Morally, it reflects some coherence in the system.  A prominent example is the quantum dilogarithms \cite{FK94}:
\begin{equation}\label{eq: pentagon for quantum dilog}
 \Phi(\hat{V}) \cdot \Phi(\hat{U}) = \Phi(\hat{U}) \cdot \Phi( q^{-1/2} \hat{V} \hat{U}) \cdot \Phi(\hat{V}), 
\end{equation}
where $\hat{V} \hat{U} = q \hat{U} \hat{V}$, and $\Phi (X) = \e \left(\sum_{k=1}^{\infty} \frac{(-1)^{k+1} X^k}{k (q^{k/2} - q^{-k/2})} \right) = \prod_{n=0}^{\infty} (1 + q^{\frac{1}{2} +n}X)$. One way to understand \eqref{eq: pentagon for quantum dilog} is that it reflects where we land if we want to commute $\Phi(\hat{U}) $ and $\Phi (\hat{V})$. 

 In \cite{GM19}, the authors obtained a pentagon relation in terms of Macdonald polynomials, generalizing \eqref{eq: pentagon for quantum dilog}. This pentagon relation is shown in \cite{Z23} to be equivalent to a pentagon relation in the elliptic Hall (or the Ding-Iohara-Miki) algebra $\cE_{q,t}$, utilizing the results in \cite{AFS} and \cite{SV13}.

The connection between skein modules and representation theory has been studied for years. 
%The HOMFLYPT skein algebra of a solid torus $\Sk(S^1 \times D)$ is isomorphic to the algebra of symmetric functions (Example \ref{eg: sk of solid t}), and the skein algebra of a torus $\Sk (T^2)$ is shown to be equivalent to be isomorphic to the $t = q$ elliptic Hall algebra $\cE_{q, q}$ (See Example \ref{eg: sk of T}, or \cite{MS17}). 
In \cite{MS21}, a surjection from the skein algebra of a punctured torus $\Sk(T-D)$ to $\cE_{q, t}$ was conjectured:
\begin{equation}
	\label{eq: sk to eha 0}
	\Sk(T-D) \twoheadrightarrow \cE_{q, t}.
\end{equation}
Combining \cite[Theorem 5.7]{MS21} and \cite[Theorem 5.10]{BCMN23}, we obtain such a map defined on a subalgebra $Sk^+(T-D)$. (See also \cite{GL23}, where this map is defined on a larger subalgebra.) 
As is pointed out in \cite[Remark 5.8]{MS21},  $\Sk(T - D)$ is much ``larger'' than $\cE_{q, t}$. Thus, it is natural to hope for a pentagon relation in $\Sk(T-D)$.

The pentagon relation also plays an important role in cluster theory. 
%Indeed, much of the magic of cluster theory can be attributed to the pentagon relation obeyed by the quantum dilogarithm. 
In a recent paper \cite{HSZ23} with Gus Schrader and Eric Zaslow, we developed a skein-theoretic version of cluster theory, and used it to compute the generating functions of Open Gromov-Witten invariants defined in \cite{ES19} (also \cite{SS23}).  One of the key constructions is to define a ``skein dilogarithm''  $Q_\x$ (or the ``Baxter operator'' in \cite{HSZ23},  see Definition \ref{def: skein dilog}), whose rank $1$ reduction recovers the usual quantum dilogarithm. %These $Q_\x$ turn out to be the image of the $T$ operators in \cite{Z23} under \eqref{eq: sk to eha 0}. 
%Indeed, the HOMFLYPT skein algebra of a surface $\Sigma$ could be viewed as the limit of ``quantized $GL_N$ characteristic variety'' (\cite{}). 
In this perspective, a pentagon relation for skeins is also desirable. 

\begin{untheorem}[Theorem \ref{thm: pentagon relation}]
	Let $\x$ and $\y$ be the $(1, 0)$ and $(0, 1)$ curves on the punctured torus $T-D$. Then in $\widehat{\Sk} (T- D)$ we have:
	\begin{equation}
		Q_\x \cdot Q_\y = Q_\y \cdot Q_{\x + \y} \cdot Q_{\x}. 
	\end{equation}
\end{untheorem}

Since on a general $2$-dimensional surface, the tubular neighborhood of any two simple closed curves is isomorphic to $T - D$, we get a pentagon relation for the skein algebra of any surface (Corollary  \ref{cor: pentagon for any surf}).

\begin{remark}
	While this paper was being written, the author was informed that this result was also proved in \cite{N24} independently. 	
\end{remark}

\subsection*{Acknowledgements}
The author would like to thank Eric Zaslow and Gus Schrader for helpful discussions and valuable suggestions on an early draft of this paper,  Peter Samuelson for sharing references, and the referee for carefully reading the paper and giving comments.

\section{HOMFLYPT Skein Modules and pentagon relation}
In this section we first recall some properties about skein modules, and then state our main theorem.  
\subsection{HOMFLYPT Skein Modules} \label{subsec: skein module}

The HOMFLYPT skein module $\Sk(M)$ of an oriented three-dimensional manifold (possibly with boundary) $M$ is generated
by $R$-linear combinations of framed oriented links in $M$, up to isotopy,
modulo the relations

	\[
	\begin{tikzpicture}[scale=0.8]
		% Overcrossing

		\node (x) at (0,0){};
		\draw[thick,->] (x.45)-- (.5,.5);
		\draw[thick,->] (x.135) -- (-.5,.5);
		\draw[thick] (x.315) -- (.5,-.5);
		\draw[thick] (x.45) -- (-.5,-.5);
		
		\node at (2,0) {$-$};
		
		% Undercrossing
		\begin{scope}[xshift=4cm]
			\node (y) at (0,0){};
			\draw[thick,->] (y.45)-- (.5,.5);
			\draw[thick,->] (y.135) -- (-.5,.5);
			\draw[thick] (y.225) -- (-.5,-.5);
			\draw[thick] (y.135) -- (.5,-.5);
			
			% \draw (-1,-1) -- (1,1);
			% \draw (1,-1)--(-1,1);
			% \draw (-1,0) node[above] {$x$};
			
			\node at (2,0) {$=$};

		\end{scope}
		
		% No crossing
		\begin{scope}[xshift=8cm]
			\node at (-.5,0) {$z$};
			\draw[thick, ->] (0,-.5) arc (-70:70:0.5);
			\draw[thick, ->] (1.2,-.5) arc (180+70:180-70:0.5);
		\end{scope}
	\end{tikzpicture}
	\]
	\phantom{a}
	\[
	\begin{tikzpicture}
		
		\node (x) at (0,0){};
		\draw[thick] (x.45) to[out=45,in=180] (.3,.2);
		\draw[thick] (.3,.2) to[out=0,in=0] (.3,-.2);
		\draw[thick] (.3,-.2) to[out=180,in=315] (x.315);
		\draw[thick] (x.45) -- (-.5,-.5);
		\draw[thick,->] (x.135) -- (-.5,.5);
		\node at (1.25,0) {$=$};
		\node at (2.15,0) {$a$};
		\draw[thick,->] (2.65,-.5)--(2.65,.5);
		\begin{scope}[xshift = 5.5cm]
			\node (y) at (0,0){};
			\draw[thick] (y.45) to[out=45,in=180] (.3,.2);
			\draw[thick] (.3,.2) to[out=0,in=0] (.3,-.2);
			\draw[thick] (.3,-.2) to[out=180,in=315] (y.315);
			\draw[thick,->] (y.315) -- (-.5,.5);
			\draw[thick] (y.225) -- (-.5,-.5);
			\node at (1.25,0) {$=$};
			\node at (2.15,0) {$a^{-1}$};
			\draw[thick,->] (2.75,-.5)--(2.75,.5);
			
		\end{scope}
		
		\begin{scope}[xshift=11cm]
			\draw[thick] (0,0) circle (.5);
			\node at (1.25,0) {$=$};
			\node at (2.5,0) {\Large $\frac{a-a^{-1}}{z}$};
		\end{scope}
		
	\end{tikzpicture}
	\]
	where $z = q^{1/2} - q^{-1/2}$
	and $R = \bC[q^{\pm 1/2},a^{\pm 1},(q^{k/2}-q^{-k/2})^{-1}]$ ,  $k$ runs over $\bN$. The above pictures are local models, and we take the convention that the framing is pointing inward.    
	\begin{example}
		It is well known that 
		\[
		\Sk(\bR^3) \simeq R \langle [\emptyset] \rangle. 
		\]
		The corresponding polynomial of a link under this identification is called the \emph{HOMFLYPT polynomial}. 
	\end{example}

	The skein algebra $\Sk(\Sigma)$ of an oriented surface $\Sigma$ is $\Sk(\Sigma\times I)$
	endowed with the product $l_1 \cdot l_2$ defined by placing the link $l_1$ above the link $l_2$:
	embed $(0,1)\sqcup (0,1) \hookrightarrow (0,2)$ then rescale $(0,2) \cong (0,1)=:I.$

	For a $3$-manifold $M$ with boundary $\partial M =  \Sigma$, the skein algebra $\Sk (\Sigma) $ has a natural left action on the skein module $\Sk(M)$, given by gluing $\Sigma \times I$ to the boundary of $M$ along $\Sigma \times \{0\}$. 
	
	\begin{theorem}[\cite{Przy98}]
		For a surface $\Sigma$, Poincare duality gives a pairing on $H_1(M, \bZ)$. We can consider the induced quantum torus $\cT_{H_1(M,\bZ)}^q$. There is a surjective map:
		\begin{equation}\label{mor: skein to quantum torus}
			\Sk(\Sigma) \twoheadrightarrow \cT_{H_1(M,\bZ)}^q
		\end{equation} 
	\end{theorem}
	Morally, the skein algebra can be considered as a ``higher rank'' version of the quantum torus. \cite{GS}
	
	\begin{example}[The skein module of the solid torus]
		\label{eg: sk of solid t} Consider the skein algebra of an annulus. As a module it will be isomorphic to the skein module of a solid torus. By results in \cite{HM06, MM08, AM98}, after choosing an orientation of \(S^1\), this skein module decomposes as:
		\begin{equation}
		\Sk(S^1 \times D) \simeq \Sk^+(S^1 \times D) \otimes_R \Sk^-(S^1 \times D).
		\end{equation}
		where \(\Sk^+(S^1 \times D)\) (resp. \(\Sk^-(S^1 \times D)\))  consists of links winding in the positive (resp. negative) direction, i.e., the tangent direction of those links, after projecting to \(S^1\), has to be positive (resp. negative) everywhere. See figure \ref{figure: solid torus} for an illustration. 
        Both \(\Sk^+(S^1 \times D)\) and $\Sk^-(S^1 \times D)$ are isomorphic to the $R$-algebra of symmetric functions with infinitely many variables: %$R [ x_1, x_2, x_3, \cdots]^{\text{Sym}}$:
		\begin{equation}		\label{eq: skein of annulus}
			\Sk^+(S^1 \times D) \simeq \Sk^-(S^1 \times D) \simeq R [ x_1, x_2, x_3, \cdots]^{\text{Sym}}. 
		\end{equation}  
	\end{example}
	
	\begin{figure}[htbp]
		\begin{tikzpicture}
			\draw[thick, red] (0,0) -- (3,0) ;
			\draw[thick] (0, 0) -- (0, 3);
			\draw[thick, red] (0, 3) -- (3,3);
			\draw[thick] (3,0) -- (3,3);
			\draw[thick, blue] (1.2, 0) .. controls ( 1.2, 1) and (1.2, 1).. (.9, 1.3);
			\draw[blue, thick, ->] (.4, 0) .. controls (.4, 1.8) and (1.2, 1.2) .. (1.2, 3);
			\draw[blue, thick, ->]  (.65, 1.52) .. controls (.4, 1.8) and (.4, 1.8).. (.4, 3);
			
			\draw[green, thick, ->] (1.8, 3) .. controls (1.8, 1.5) and (2.6, 1.5) .. (2.6, 0);
			\draw[green, thick, ] (2.6, 3).. controls (2.6, 1.9) .. (2.25, 1.7);
			\draw[green, thick, ->] (2.0, 1.6).. controls (1.8, 1.5) .. (1.8, 0);
		\end{tikzpicture}
		
		\caption{A skein element in \(\Sk(S^1 \times D^2)\). The red edges are identified so we get a link in a solid torus. We take the convention that going upward is the positive direction. Hence, the blue part lives in $\Sk^+(S^1 \times D)$, while the green part lives in \(\Sk^-(S^1 \times D)\). }
 		\label{figure: solid torus}
	\end{figure}
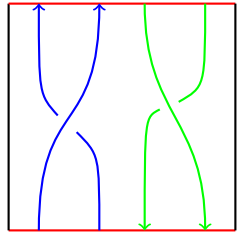
	
	\begin{notation}\label{notation: nl}

	For an oriented framed knot $l$ in a $3$-manifold $M$, we use $P_l$ to denote the element it represents in $\Sk(M)$. Note that the tubular neighborhood of \(l\) is homeomorphic to a solid torus, and that a framing is just a choice of the longitude, which will give an identification of the solid torus with \(S^1 \times D^2\) (up to isotopy). 
    Thus we have a natural $R$-module morphism induced by $l$:
	\[
	\Sk^+(S^1 \times D) \longrightarrow \Sk(M). 
	\]
	We denote by $P_{nl}$ the image of the skein element corresponding to the symmetric polynomial $p_n = \sum_i x_i^n \in R [ x_1, x_2, x_3, \cdots]^{\text{Sym}}$ under \eqref{eq: skein of annulus}. 
			
	\end{notation}
	
	\begin{example}[\cite{MS17} The skein algebra of the torus]
		\label{eg: sk of T}
		Let $T = \bR^2 / \bZ^2$ be the torus. The skein algebra $\Sk(T )$ is isomorphic to the $R$ algebra generated by $\{P_\x\}, \x \in \bZ^2$ modulo the relation:
		\begin{equation}
		[P_\x, P_\y]  = \{d\} P_{\x + \y},
		\label{eq: relation of skein of the torus}
		\end{equation}
		where $d = \det(\x|\y)$ and $\{d\} = q^{d/2} - q^{-d/2}$. 
		For a primitive $\x \in \bZ^2$, $P_\x$ is defined to be the torus knot with slope $\x$, and $P_{k\x}$ is defined as as in Notation \ref{notation: nl}.  In particular, $\Sk(T )$ is isomorphic to the $q = t$ specialized elliptic Hall algebra $\cE_{q, q}$. See \cite{MS17} for details.
	\end{example}
	
	\begin{example}[The skein algebra of the punctured torus]\label{eg: skein of punctured torus}
		Consider then skein algebra of a torus minus a disk $\Sk(T - D)$. The elements $P_\x, \ \x \in \bZ^2$, are still defined. However,  the relation \eqref{eq: relation of skein of the torus} only holds when $\x = k \x_0$, and $\det(\x_0 | \y) =\pm 1$. See \cite[Remark 4.3]{MS21}. 
	\end{example}
	
	Define $\mathbf{Z}^+ \coloneqq \{ (a, b) \in \bZ^2 |\ a>0 \ \text{or}\  a=0, b\ge 0\}$. Let $\Sk^+(T-D)$ be the subalgebra of $\Sk(T - D)$ generated by $P_\x$ for $\x \in \mathbf{Z}^+$.  
	
	\begin{theorem}[\cite{MS21, BCMN23}]
		There is a surjective $R$-algebra map  from $\Sk^+(T  - D) $ to the positive part of the elliptic hall algebra:
		\begin{equation}
			\label{mor: skein to eha}
			\Sk^+(T  - D) \twoheadrightarrow \cE^+_{q,t}
		\end{equation}
	\end{theorem}
	Recall that the elliptic Hall algebra $\cE_{q, t}$ is a $\bC(q, t)$-algebra generated by ${u_\x}, \x \in \bZ^2$, modulo some relations. $\cE_{q, t}^+$ is the subalgebra generated by $u_\x$ for $\x \in \mathbf{Z}^+$. The map \eqref{mor: skein to eha} is defined by:
	\[
	P_{k \x_0} \longmapsto \{k\}u_{k \x_0} 
	\] 
	for primitive $\x_0$. See \cite[Theorem 5.7]{MS21} for details.

	\subsection{Skein Dilogarithms and the Pentagon Relation}\label{subsec: sk dilog}
	
	We first recall the definition of the completion of skein modules in \cite{HSZ23}. 
	
	The skein module $\Sk (M)$ is graded
	by the group $H_1 (M; \mathbb{Z})$. Write $\Sk(M)_\lambda$ for the $\lambda$-graded piece,  $\lambda \in H_1(M, \bZ)$.   If we choose a homomorphism 
	\begin{equation}\label{eq: completion morphism}
	\delta: \quad H_1 (M; \bZ) \longrightarrow \bZ
	\end{equation}
	then  \( \left\{\bigoplus_{ \delta (\lambda) \le n}  \Sk (M)_\lambda \right\}_{n\in \bN} \) gives a filtration of \(\Sk(M)\). Therefore we can define the completion to be:
	\[
	\widehat{\Sk} (M) \coloneqq \varprojlim_{n\ge 0} \left( \bigoplus_{ \delta (\lambda) \le n}  \Sk (M)_\lambda  \right).
	\]

	For a surface $\Sigma$, the completion $\widehat{\Sk}(\Sigma) \coloneqq
	\widehat{\Sk}(\Sigma \times I)$ retains the structure of an algebra.  Moreover, if $\partial M = \Sigma$ and we choose compatible homomorphisms to $\bZ$,
	i.e.~commuting with the natural map $H_1(\Sigma; \bZ) \rightarrow H_1 (M; \bZ)$, then $\widehat{\Sk}(M)$ inherits
	a module structure over $\widehat {\Sk} (\Sigma)$.

    \begin{example}	\label{eg: completion}

        \begin{figure}[ht]
			\centering

		\tikzset{every picture/.style={line width=0.75pt}} %set default line width to 0.75pt        
		\begin{tikzpicture}[x=0.75pt,y=0.75pt,yscale=-.6,xscale=.6]
			%uncomment if require: \path (0,234); %set diagram left start at 0, and has height of 234
			
			%Curve Lines [id:da9503904552711944] 
			\draw    (428.2,58.5) .. controls (449.2,49.5) and (488.2,43.5) .. (508.2,53.5) .. controls (528.2,63.5) and (549.2,93.5) .. (545.2,121.5) .. controls (541.2,149.5) and (512.21,157.21) .. (481.2,161.5) .. controls (450.19,165.79) and (420.74,153.56) .. (380.2,150.5) .. controls (339.66,147.44) and (317.61,152.9) .. (310.27,153.99) .. controls (302.93,155.08) and (269.42,163.15) .. (240.2,165.5) .. controls (210.98,167.85) and (186.2,165.5) .. (166.2,157.5) .. controls (146.2,149.5) and (134.2,128.5) .. (134.2,108.5) .. controls (134.2,88.5) and (135.2,75.5) .. (150.2,60.5) .. controls (165.2,45.5) and (204.41,42.85) .. (245.2,45.5) .. controls (285.99,48.15) and (315.2,61.5) .. (349.2,64.5) .. controls (383.2,67.5) and (415.2,65.5) .. (428.2,58.5) -- cycle ;
			%Curve Lines [id:da3492572169098038] 
			\draw    (168.2,104.5) .. controls (177.8,118.31) and (187.9,121.1) .. (202.25,118.84) .. controls (216.61,116.58) and (223.6,108.69) .. (229.2,101.5) ;
			%Curve Lines [id:da26547478596751883] 
			\draw    (179.2,109.5) .. controls (199.2,96.5) and (208.4,97) .. (219.2,105.5) ;
			%Curve Lines [id:da26181451678081635] 
			\draw    (270.2,103.5) .. controls (279.8,117.31) and (289.9,120.1) .. (304.25,117.84) .. controls (318.61,115.58) and (325.6,107.69) .. (331.2,100.5) ;
			%Curve Lines [id:da4332540547945727] 
			\draw    (281.2,108.5) .. controls (301.2,95.5) and (310.4,96) .. (321.2,104.5) ;
			%Curve Lines [id:da8738652404177669] 
			\draw    (442.2,104.5) .. controls (451.8,118.31) and (461.9,121.1) .. (476.25,118.84) .. controls (490.61,116.58) and (497.6,108.69) .. (503.2,101.5) ;
			%Curve Lines [id:da06040629659137298] 
			\draw    (453.2,109.5) .. controls (473.2,96.5) and (482.4,97) .. (493.2,105.5) ;
			%Shape: Arc [id:dp5979905705814914] 
			\draw  [draw opacity=0] (198.72,83.48) .. controls (199.17,83.47) and (199.62,83.47) .. (200.07,83.47) .. controls (222.56,83.47) and (240.8,94.17) .. (240.8,107.37) .. controls (240.8,120.57) and (222.56,131.27) .. (200.07,131.27) .. controls (177.57,131.27) and (159.33,120.57) .. (159.33,107.37) .. controls (159.33,103.03) and (161.3,98.97) .. (164.73,95.47) -- (200.07,107.37) -- cycle ; \draw [color={rgb, 255:red, 208; green, 2; blue, 27 }  ,draw opacity=1 ]   (200.07,83.47) .. controls (222.56,83.47) and (240.8,94.17) .. (240.8,107.37) .. controls (240.8,120.57) and (222.56,131.27) .. (200.07,131.27) .. controls (177.57,131.27) and (159.33,120.57) .. (159.33,107.37) .. controls (159.33,103.03) and (161.3,98.97) .. (164.73,95.47) ;  \draw [shift={(198.72,83.48)}, rotate = 3.72] [fill={rgb, 255:red, 208; green, 2; blue, 27 }  ,fill opacity=1 ][line width=0.08]  [draw opacity=0] (10.72,-5.15) -- (0,0) -- (10.72,5.15) -- (7.12,0) -- cycle    ;
			%Shape: Arc [id:dp28846469120862506] 
			\draw  [draw opacity=0] (238.74,99.84) .. controls (240.08,102.2) and (240.8,104.74) .. (240.8,107.37) .. controls (240.8,120.57) and (222.56,131.27) .. (200.07,131.27) .. controls (177.57,131.27) and (159.33,120.57) .. (159.33,107.37) .. controls (159.33,94.17) and (177.57,83.47) .. (200.07,83.47) .. controls (209.45,83.47) and (218.1,85.33) .. (224.99,88.46) -- (200.07,107.37) -- cycle ; \draw [color={rgb, 255:red, 208; green, 2; blue, 27 }  ,draw opacity=1 ]   (238.74,99.84) .. controls (240.08,102.2) and (240.8,104.74) .. (240.8,107.37) .. controls (240.8,120.57) and (222.56,131.27) .. (200.07,131.27) .. controls (177.57,131.27) and (159.33,120.57) .. (159.33,107.37) .. controls (159.33,94.17) and (177.57,83.47) .. (200.07,83.47) .. controls (209.45,83.47) and (218.1,85.33) .. (224.99,88.46) ;  
			%Shape: Arc [id:dp31853545688222384] 
			\draw  [draw opacity=0] (299.39,82.15) .. controls (299.83,82.14) and (300.28,82.13) .. (300.73,82.13) .. controls (323.23,82.13) and (341.47,92.83) .. (341.47,106.03) .. controls (341.47,119.23) and (323.23,129.93) .. (300.73,129.93) .. controls (278.24,129.93) and (260,119.23) .. (260,106.03) .. controls (260,101.7) and (261.96,97.64) .. (265.4,94.14) -- (300.73,106.03) -- cycle ; \draw [color={rgb, 255:red, 208; green, 2; blue, 27 }  ,draw opacity=1 ]   (300.73,82.13) .. controls (323.23,82.13) and (341.47,92.83) .. (341.47,106.03) .. controls (341.47,119.23) and (323.23,129.93) .. (300.73,129.93) .. controls (278.24,129.93) and (260,119.23) .. (260,106.03) .. controls (260,101.7) and (261.96,97.64) .. (265.4,94.14) ;  \draw [shift={(299.39,82.15)}, rotate = 3.72] [fill={rgb, 255:red, 208; green, 2; blue, 27 }  ,fill opacity=1 ][line width=0.08]  [draw opacity=0] (10.72,-5.15) -- (0,0) -- (10.72,5.15) -- (7.12,0) -- cycle    ;
			%Shape: Arc [id:dp6303101980924601] 
			\draw  [draw opacity=0] (339.4,98.5) .. controls (340.74,100.87) and (341.47,103.4) .. (341.47,106.03) .. controls (341.47,119.23) and (323.23,129.93) .. (300.73,129.93) .. controls (278.24,129.93) and (260,119.23) .. (260,106.03) .. controls (260,92.83) and (278.24,82.13) .. (300.73,82.13) .. controls (310.12,82.13) and (318.76,84) .. (325.65,87.13) -- (300.73,106.03) -- cycle ; \draw [color={rgb, 255:red, 208; green, 2; blue, 27 }  ,draw opacity=1 ]   (339.4,98.5) .. controls (340.74,100.87) and (341.47,103.4) .. (341.47,106.03) .. controls (341.47,119.23) and (323.23,129.93) .. (300.73,129.93) .. controls (278.24,129.93) and (260,119.23) .. (260,106.03) .. controls (260,92.83) and (278.24,82.13) .. (300.73,82.13) .. controls (310.12,82.13) and (318.76,84) .. (325.65,87.13) ;  
			%Shape: Arc [id:dp8831307575553378] 
			\draw  [draw opacity=0] (471.39,82.81) .. controls (471.83,82.8) and (472.28,82.8) .. (472.73,82.8) .. controls (495.23,82.8) and (513.47,93.5) .. (513.47,106.7) .. controls (513.47,119.9) and (495.23,130.6) .. (472.73,130.6) .. controls (450.24,130.6) and (432,119.9) .. (432,106.7) .. controls (432,102.37) and (433.96,98.31) .. (437.4,94.8) -- (472.73,106.7) -- cycle ; \draw [color={rgb, 255:red, 208; green, 2; blue, 27 }  ,draw opacity=1 ]   (472.73,82.8) .. controls (495.23,82.8) and (513.47,93.5) .. (513.47,106.7) .. controls (513.47,119.9) and (495.23,130.6) .. (472.73,130.6) .. controls (450.24,130.6) and (432,119.9) .. (432,106.7) .. controls (432,102.37) and (433.96,98.31) .. (437.4,94.8) ;  \draw [shift={(471.39,82.81)}, rotate = 3.72] [fill={rgb, 255:red, 208; green, 2; blue, 27 }  ,fill opacity=1 ][line width=0.08]  [draw opacity=0] (10.72,-5.15) -- (0,0) -- (10.72,5.15) -- (7.12,0) -- cycle    ;
			%Shape: Arc [id:dp7647376234650483] 
			\draw  [draw opacity=0] (511.4,99.17) .. controls (512.74,101.54) and (513.47,104.07) .. (513.47,106.7) .. controls (513.47,119.9) and (495.23,130.6) .. (472.73,130.6) .. controls (450.24,130.6) and (432,119.9) .. (432,106.7) .. controls (432,93.5) and (450.24,82.8) .. (472.73,82.8) .. controls (482.12,82.8) and (490.76,84.66) .. (497.65,87.79) -- (472.73,106.7) -- cycle ; \draw [color={rgb, 255:red, 208; green, 2; blue, 27 }  ,draw opacity=1 ]   (511.4,99.17) .. controls (512.74,101.54) and (513.47,104.07) .. (513.47,106.7) .. controls (513.47,119.9) and (495.23,130.6) .. (472.73,130.6) .. controls (450.24,130.6) and (432,119.9) .. (432,106.7) .. controls (432,93.5) and (450.24,82.8) .. (472.73,82.8) .. controls (482.12,82.8) and (490.76,84.66) .. (497.65,87.79) ;  
			
			% Text Node
			\draw (365,95) node [anchor=north west][inner sep=0.75pt]  [font=\Large] [align=left] {$\displaystyle \cdots $};
			% Text Node
			\draw (207.6,145.75) node  [color={rgb, 255:red, 208; green, 2; blue, 27 }  ,opacity=1 ] [align=left] {\begin{minipage}[lt]{23.94pt}\setlength\topsep{0pt}
					$\displaystyle a_{1}$
			\end{minipage}};
			% Text Node
			\draw (305.27,142.42) node  [color={rgb, 255:red, 208; green, 2; blue, 27 }  ,opacity=1 ] [align=left] {\begin{minipage}[lt]{23.94pt}\setlength\topsep{0pt}
					$\displaystyle a_{2}$
			\end{minipage}};
			% Text Node
			\draw (479.6,142.75) node  [color={rgb, 255:red, 208; green, 2; blue, 27 }  ,opacity=1 ] [align=left] {\begin{minipage}[lt]{23.94pt}\setlength\topsep{0pt}
					$\displaystyle a_{g}$
			\end{minipage}};

		\end{tikzpicture}
		
			\caption{Loops $a_i$ on $\Sigma_g$}
			\label{fig: Sigma_g}

		\end{figure}
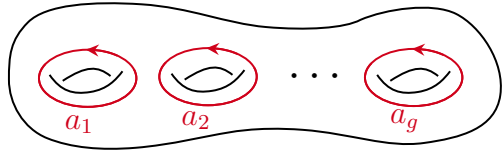

        We will focus on the genus-$g$ handlebody $\cH_g$ with boundary $\Sigma_g$ --- see Figure \ref{fig: Sigma_g}.      
	The $a_i$ circles form a basis for $H_1 (\cH_g, \bZ)$. We then     define the homomorphism for $\cH_g$ to be
		\begin{align*}
			\delta: H_1 (\cH_g; \bZ) &\longrightarrow \bZ \\
			\sum_{i=1}^g c_i [a_i] &\longmapsto \sum_{i=1}^g c_i,
		\end{align*}  
		and for $\Sigma_g$ to be $\delta$ composed with the natural map $H_1(\Sigma_g; \bZ) \rightarrow H_1( \cH_g ; \bZ)$. The corresponding completion $\widehat{\Sk} (\cH_g) $ is a well-defined module over the algebra $\widehat{\Sk} (\Sigma_g)$. 

    \end{example}

	We say a loop is \textit{positive}/\textit{negative} with respect to $\delta$ if the image of its homology class under $\delta$ is positive/negative. 
	
	\begin{definition}[{\cite[Definition 5.8]{HSZ23}}]
		\label{def: skein dilog}
	For a positive loop $l$ on $\Sigma$, we define the \emph{Skein Dilogarithm} to be 
	\begin{equation}
		Q_l (u) \coloneqq \e \left(\sum_{k \ge 0} \frac{(-1)^{k+1} }{  k \{k\}} u^k P_{kl} \right)  \in \widehat{\Sk}(\Sigma),
	\end{equation}
	where $\{k\} = q^{\frac{k}{2} } -q^{-\frac{k}{2}}$, $u$ is an element in $R$, and $P_{nl}$ are defined in Section \ref{subsec: skein module}. 
	We write $Q_l$ for $Q_{l} (1) $ for short. 		
	\end{definition}
	Under the map \eqref{mor: skein to quantum torus}, $Q_l$ is mapped to the quantum dilogarithm $\Phi(X_{[l]}) \in \widehat{\cT}^q_{H_1(\Sigma, \bZ)}$. Here \(X_{[l]}\) is the quantum torus generator associated to the homology class of \(l\). 
    
	\begin{notation}\label{notation: x+l}
	
	Let $\x$ and $l$ be simple closed curves on $\Sigma$, transversely intersecting at a single point. Use $\x + l$ to denote their connected sum.
		
\begin{center}

		\tikzset{every picture/.style={line width=0.75pt}} %set default line width to 0.75pt        
		
		\begin{tikzpicture}[x=0.75pt,y=0.75pt,yscale=-1,xscale=1]
			%uncomment if require: \path (0,216); %set diagram left start at 0, and has height of 216
			
			%Straight Lines [id:da2267373778790649] 
			\draw [color={rgb, 255:red, 74; green, 144; blue, 226 }  ,draw opacity=1 ][line width=1.5]    (162.39,113.13) -- (162.39,24.53) ;
			\draw [shift={(162.39,20.53)}, rotate = 90] [fill={rgb, 255:red, 74; green, 144; blue, 226 }  ,fill opacity=1 ][line width=0.08]  [draw opacity=0] (13.4,-6.43) -- (0,0) -- (13.4,6.44) -- (8.9,0) -- cycle    ;
			%Curve Lines [id:da8969214072689562] 
			\draw [color={rgb, 255:red, 189; green, 16; blue, 224 }  ,draw opacity=1 ][line width=1.5]    (273.42,66.89) .. controls (308.52,66.96) and (318.82,67.12) .. (319.67,23.99) ;
			\draw [shift={(319.72,20.59)}, rotate = 90.48] [fill={rgb, 255:red, 189; green, 16; blue, 224 }  ,fill opacity=1 ][line width=0.08]  [draw opacity=0] (13.4,-6.43) -- (0,0) -- (13.4,6.44) -- (8.9,0) -- cycle    ;
			%Straight Lines [id:da9200309840125365] 
			\draw [line width=1.5]    (116.09,41.54) -- (116.09,92.12) ;
			%Shape: Arc [id:dp5710747526880444] 
			\draw  [draw opacity=0][dash pattern={on 1.69pt off 2.76pt}][line width=1.5]  (140.55,20.53) .. controls (140.55,32.14) and (129.6,41.54) .. (116.09,41.54) -- (116.09,20.53) -- cycle ; \draw  [dash pattern={on 1.69pt off 2.76pt}][line width=1.5]  (140.55,20.53) .. controls (140.55,32.14) and (129.6,41.54) .. (116.09,41.54) ;  
			%Straight Lines [id:da8500370058163331] 
			\draw [line width=1.5]    (184.23,20.53) -- (140.55,20.53) ;
			%Shape: Arc [id:dp0052280776579727295] 
			\draw  [draw opacity=0][dash pattern={on 1.69pt off 2.76pt}][line width=1.5]  (184.23,113.13) .. controls (184.23,101.53) and (195.18,92.12) .. (208.69,92.12) -- (208.69,113.13) -- cycle ; \draw  [dash pattern={on 1.69pt off 2.76pt}][line width=1.5]  (184.23,113.13) .. controls (184.23,101.53) and (195.18,92.12) .. (208.69,92.12) ;  
			%Shape: Arc [id:dp041088697475036184] 
			\draw  [draw opacity=0][dash pattern={on 1.69pt off 2.76pt}][line width=1.5]  (208.69,41.54) .. controls (195.18,41.54) and (184.23,32.14) .. (184.23,20.53) -- (208.69,20.53) -- cycle ; \draw  [dash pattern={on 1.69pt off 2.76pt}][line width=1.5]  (208.69,41.54) .. controls (195.18,41.54) and (184.23,32.14) .. (184.23,20.53) ;  
			%Shape: Arc [id:dp2643178791230143] 
			\draw  [draw opacity=0][dash pattern={on 1.69pt off 2.76pt}][line width=1.5]  (116.09,92.12) .. controls (129.6,92.12) and (140.55,101.53) .. (140.55,113.13) -- (116.09,113.13) -- cycle ; \draw  [dash pattern={on 1.69pt off 2.76pt}][line width=1.5]  (116.09,92.12) .. controls (129.6,92.12) and (140.55,101.53) .. (140.55,113.13) ;  
			%Straight Lines [id:da4443161076355975] 
			\draw [line width=1.5]    (208.69,41.54) -- (208.69,92.12) ;
			%Straight Lines [id:da32544288029556245] 
			\draw [line width=1.5]    (184.23,113.13) -- (140.55,113.13) ;
			%Straight Lines [id:da5963467210573328] 
			\draw [line width=1.5]    (273.42,41.6) -- (273.42,92.18) ;
			%Shape: Arc [id:dp020236618322760025] 
			\draw  [draw opacity=0][dash pattern={on 1.69pt off 2.76pt}][line width=1.5]  (297.88,20.59) .. controls (297.88,32.19) and (286.93,41.6) .. (273.42,41.6) .. controls (273.42,41.6) and (273.42,41.6) .. (273.42,41.6) -- (273.42,20.59) -- cycle ; \draw  [dash pattern={on 1.69pt off 2.76pt}][line width=1.5]  (297.88,20.59) .. controls (297.88,32.19) and (286.93,41.6) .. (273.42,41.6) .. controls (273.42,41.6) and (273.42,41.6) .. (273.42,41.6) ;  
			%Straight Lines [id:da001927806984995506] 
			\draw [line width=1.5]    (341.56,20.59) -- (297.88,20.59) ;
			%Shape: Arc [id:dp1959737797646559] 
			\draw  [draw opacity=0][dash pattern={on 1.69pt off 2.76pt}][line width=1.5]  (341.56,113.19) .. controls (341.56,101.59) and (352.51,92.18) .. (366.02,92.18) -- (366.02,113.19) -- cycle ; \draw  [dash pattern={on 1.69pt off 2.76pt}][line width=1.5]  (341.56,113.19) .. controls (341.56,101.59) and (352.51,92.18) .. (366.02,92.18) ;  
			%Shape: Arc [id:dp9354034499911121] 
			\draw  [draw opacity=0][dash pattern={on 1.69pt off 2.76pt}][line width=1.5]  (366.02,41.6) .. controls (352.51,41.6) and (341.56,32.19) .. (341.56,20.59) -- (366.02,20.59) -- cycle ; \draw  [dash pattern={on 1.69pt off 2.76pt}][line width=1.5]  (366.02,41.6) .. controls (352.51,41.6) and (341.56,32.19) .. (341.56,20.59) ;  
			%Shape: Arc [id:dp7712036483867126] 
			\draw  [draw opacity=0][dash pattern={on 1.69pt off 2.76pt}][line width=1.5]  (273.42,92.18) .. controls (286.93,92.18) and (297.88,101.59) .. (297.88,113.19) -- (273.42,113.19) -- cycle ; \draw  [dash pattern={on 1.69pt off 2.76pt}][line width=1.5]  (273.42,92.18) .. controls (286.93,92.18) and (297.88,101.59) .. (297.88,113.19) ;  
			%Straight Lines [id:da31543653963925244] 
			\draw [line width=1.5]    (366.02,41.6) -- (366.02,92.18) ;
			%Straight Lines [id:da7767326028448496] 
			\draw [line width=1.5]    (341.56,113.19) -- (297.88,113.19) ;
			%Curve Lines [id:da3802833691290535] 
			\draw [color={rgb, 255:red, 189; green, 16; blue, 224 }  ,draw opacity=1 ][line width=1.5]    (319.72,113.19) .. controls (320.64,69.16) and (320.67,67.84) .. (362.06,66.97) ;
			\draw [shift={(366.02,66.89)}, rotate = 178.85] [fill={rgb, 255:red, 189; green, 16; blue, 224 }  ,fill opacity=1 ][line width=0.08]  [draw opacity=0] (13.4,-6.43) -- (0,0) -- (13.4,6.44) -- (8.9,0) -- cycle    ;
			%Straight Lines [id:da6160393578449852] 
			\draw [color={rgb, 255:red, 208; green, 2; blue, 27 }  ,draw opacity=1 ][line width=1.5]    (116.09,66.83) -- (204.69,66.83) ;
			\draw [shift={(208.69,66.83)}, rotate = 180] [fill={rgb, 255:red, 208; green, 2; blue, 27 }  ,fill opacity=1 ][line width=0.08]  [draw opacity=0] (13.4,-6.43) -- (0,0) -- (13.4,6.44) -- (8.9,0) -- cycle    ;
			
			% Text Node
			\draw (181.27,70.47) node [anchor=north west][inner sep=0.75pt]  [color={rgb, 255:red, 208; green, 2; blue, 27 }  ,opacity=1 ] [align=left] {$\displaystyle \x$};
			% Text Node
			\draw (146.8,33.1) node [anchor=north west][inner sep=0.75pt]  [color={rgb, 255:red, 74; green, 144; blue, 226 }  ,opacity=1 ] [align=left] {$\displaystyle l$};
			% Text Node
			\draw (276.67,65.67) node [anchor=north west][inner sep=0.75pt]  [color={rgb, 255:red, 189; green, 16; blue, 224 }  ,opacity=1 ] [align=left] {$\displaystyle \x+l$};
		
		\end{tikzpicture}
		
	\end{center}

\end{notation}

We have:
\begin{proposition}[{\cite[Lemma 5.5]{HSZ23}}]
	\label{prop: Ad action}
	\begin{align*}
		\Ad_{Q_\x} P_l &= P_l + P_{\x+l}, \quad \text{if \( \x \) intersects \( l \) positively, and}\\
		\Ad_{Q_\x^{-1}} P_l& = P_l + P_{\x+l} \quad \text{if \( \x \) intersects \( l \) negatively.}
	\end{align*}
\end{proposition}

Now we can state our main theorem.

\begin{theorem}\label{thm: pentagon relation}
Consider the skein algebra of the punctured torus \(\Sk (T -D)\). We take the  morphism \eqref{eq: completion morphism} to be 
\begin{align*}
    H_1(T - D) \simeq \bZ^2 &\longrightarrow \bZ \\
    (i, j) &\longmapsto i+j,
\end{align*}
hence get a completion \(\widehat{\Sk} (T - D)\). 

Let \(\x\) and \(\y\) be the \((1, 0)\) and \((0, 1)\) loops, respectively, and \(v\), \(w\) be elements in \(R\). Then we have:
\begin{equation}\label{eq: pentagon relation}
    Q_\x(v) \cdot Q_\y(w) = Q_{\y}(w) \cdot Q_{\x + \y}(vw) \cdot Q_{\x}(v).
\end{equation}
\end{theorem}	

\begin{remark}
All terms involved in \eqref{eq: pentagon relation} live in \(\Sk^+(T- D)\). One can check that under \eqref{mor: skein to eha}, our pentagon relation is mapped to that in \cite{Z23}, which in \cite{Z23} is also shown to be equivalent to the ``\(5\)-term relation'' of \cite{GM19}.
\end{remark}

Note that \(\widehat{\Sk}(T - D)  \) is naturally graded by \(H_1(T - D)\simeq \bZ^2\), and in \eqref{eq: pentagon relation} the \(v\) and  \(w\) variables are just to record the degree. 
More precisely, if we expand everything, the equation \eqref{eq: pentagon relation} will become
\begin{align*}
\sum_{(i, j) \in {\bN}^2} v^i w^j \cdot & \left( \text{the degree $(i, j)$ part of }  Q_\x \cdot Q_\y \right) \\
&= \sum_{(i, j) \in {\bN}^2} v^i w^j \cdot \left( \text{the degree $(i, j)$ part of } Q_\y \cdot Q_{\x + \y} \cdot Q_\x \right) 
\end{align*}
The above holds if and only if each degree \( (i, j)\) part matches up for all \(i\) and \(j\), which is equivalent to 
\[
Q_\x \cdot Q_\y = Q_\y \cdot Q_{\x + \y} \cdot Q_\x.
\]
In other words, we only need to prove the \(v = w =1\) case.

We want to point out that the equation \eqref{eq: pentagon relation} is already nontrivial in low degree terms. For example, if we denote \(P'_{k\x} = \frac{1}{\{k\}} P_\x\) for primitive \(\x\), the degree \((2, 2)\) part (i.e. coefficient of term \(v^2 w^2\)) is equivalent to:
\[
[P'_{2\x}, P'_{2\y}] = [P'_{\x}, P'_{\x + 2\y}] - 2 P'_{2\x + 2\y},
\]
We don't know any reason for this relation to hold a priori, though one can verify it by a direct but tedious skein diagram computation. 

For a general surface \(\Sigma\), let \(l_1\) and \(l_2\) still be two simple closed curves intersecting positively in a single point. Since their neighborhood is homeomorphic to \(T- D\), we have:
\begin{corollary}\label{cor: pentagon for any surf}
\begin{equation*}
    Q_{l_1} (v) \cdot Q_{l_2} (w) = Q_{l_2} (w) \cdot Q_{l_1 + l_2} (wv)\cdot  Q_{l_1} (v)
\end{equation*}
holds in \(\widehat{\Sk}(\Sigma)\). 
\end{corollary}

\section{proof of the main theorem}

In this section we prove \eqref{eq: pentagon relation}. As mentioned in last section, we only need to prove the case when \(v = w =1\).

We recall a key lemma from \cite{HSZ23}.

\begin{lemma}[{\cite[Lemma 7.2]{HSZ23}}] \label{lem: sliding}
	 Let $M$ be a $3$-manifold with boundary $\Sigma$, and $\gamma$ a loop on $\Sigma$ bounding a disk $D$ in $M$, with framing pointing inward along \(D\). If an element  $\Psi \in {\Sk}(M)$ satisfies the ``sliding relation'':
	\begin{equation}\label{eq: sliding relation}
		(\skein{\gamma} - \bigcirc) \Psi = 0,
	\end{equation}
	then  $\Psi$ lives in the image of ${\Sk} (M \backslash D)\rightarrow {\Sk}(M).$ 
    Here \(\bigcirc\) is the unknot. 
\end{lemma}

We first reduce the pentagon relation for $\Sk (T - D)$ to an equation in the skein module of the genus $2$ handlebody $\Sk(\cH_2)$, and then use the above lemma to prove it. 

\begin{figure}[htbp]

			\centering		
			\tikzset{every picture/.style={line width=0.75pt}} %set default line width to 0.75pt        
			
			\begin{tikzpicture}[x=0.75pt,y=0.75pt,yscale=-1,xscale=1]
				%uncomment if require: \path (0,167); %set diagram left start at 0, and has height of 167
				
				%Shape: Polygon Curved [id:ds218487706372962] 
				\draw   (127.2,53.6) .. controls (147.2,43.6) and (177.87,54.93) .. (203.87,56.93) .. controls (229.87,58.93) and (233.2,53.6) .. (254.53,49.6) .. controls (275.87,45.6) and (291.87,53.6) .. (299.87,66.93) .. controls (307.87,80.27) and (303.87,113.6) .. (284.53,122.27) .. controls (265.2,130.93) and (254.53,120.93) .. (209.87,120.27) .. controls (165.2,119.6) and (132.53,137.6) .. (113.87,114.93) .. controls (95.2,92.27) and (107.2,63.6) .. (127.2,53.6) -- cycle ;
				%Curve Lines [id:da7478127063213136] 
				\draw    (141.2,86.93) .. controls (153.87,101.6) and (175.87,98.27) .. (185.2,85.6) ;
				%Curve Lines [id:da9177755644606174] 
				\draw    (148.53,88.93) .. controls (153.87,80.93) and (169.2,80.27) .. (177.2,87.6) ;
				%Curve Lines [id:da6757253043220481] 
				\draw    (217.87,85.6) .. controls (230.53,100.27) and (252.53,96.93) .. (261.87,84.27) ;
				%Curve Lines [id:da19146133550108413] 
				\draw    (225.2,87.6) .. controls (230.53,79.6) and (245.87,78.93) .. (253.87,86.27) ;
				%Shape: Ellipse [id:dp45134601307746647] 
				\draw  [color={rgb, 255:red, 208; green, 2; blue, 27 }  ,draw opacity=1 ] (135.2,87.97) .. controls (135.2,78.23) and (147.89,70.33) .. (163.53,70.33) .. controls (179.18,70.33) and (191.87,78.23) .. (191.87,87.97) .. controls (191.87,97.71) and (179.18,105.6) .. (163.53,105.6) .. controls (147.89,105.6) and (135.2,97.71) .. (135.2,87.97) -- cycle ;
				%Straight Lines [id:da5978505851444627] 
				\draw [color={rgb, 255:red, 208; green, 2; blue, 27 }  ,draw opacity=1 ]   (162.27,70.33) -- (169.87,66.93) ;
				%Straight Lines [id:da01930187519969717] 
				\draw [color={rgb, 255:red, 208; green, 2; blue, 27 }  ,draw opacity=1 ]   (169.2,74.93) -- (162.27,70.33) ;
				%Shape: Ellipse [id:dp8063747698997654] 
				\draw  [color={rgb, 255:red, 208; green, 2; blue, 27 }  ,draw opacity=1 ] (209.2,85.91) .. controls (209.2,75.77) and (223.23,67.56) .. (240.53,67.56) .. controls (257.84,67.56) and (271.87,75.77) .. (271.87,85.91) .. controls (271.87,96.05) and (257.84,104.27) .. (240.53,104.27) .. controls (223.23,104.27) and (209.2,96.05) .. (209.2,85.91) -- cycle ;
				%Straight Lines [id:da10537298314920229] 
				\draw [color={rgb, 255:red, 208; green, 2; blue, 27 }  ,draw opacity=1 ]   (240.53,67.56) -- (247.62,64.27) ;
				%Straight Lines [id:da7854133981884] 
				\draw [color={rgb, 255:red, 208; green, 2; blue, 27 }  ,draw opacity=1 ]   (247,72.01) -- (240.53,67.56) ;
				%Curve Lines [id:da0799278806027226] 
				\draw [color={rgb, 255:red, 74; green, 144; blue, 226 }  ,draw opacity=1 ]   (105.2,85.6) .. controls (105.1,64.46) and (151.6,64.46) .. (150.53,85.6) ;
				\draw [shift={(121.75,70.31)}, rotate = 354.41] [fill={rgb, 255:red, 74; green, 144; blue, 226 }  ,fill opacity=1 ][line width=0.08]  [draw opacity=0] (10.72,-5.15) -- (0,0) -- (10.72,5.15) -- (7.12,0) -- cycle    ;
				%Curve Lines [id:da24417304327003642] 
				\draw [color={rgb, 255:red, 74; green, 144; blue, 226 }  ,draw opacity=1 ] [dash pattern={on 4.5pt off 4.5pt}]  (105.2,85.6) .. controls (109.2,100.93) and (141.2,98.27) .. (150.53,85.6) ;
				%Curve Lines [id:da863920984799575] 
				\draw [color={rgb, 255:red, 74; green, 144; blue, 226 }  ,draw opacity=1 ]   (181.87,86.27) .. controls (180.1,67.46) and (222.1,63.46) .. (222.53,89.6) ;
				\draw [shift={(197.08,71.44)}, rotate = 356.77] [fill={rgb, 255:red, 74; green, 144; blue, 226 }  ,fill opacity=1 ][line width=0.08]  [draw opacity=0] (10.72,-5.15) -- (0,0) -- (10.72,5.15) -- (7.12,0) -- cycle    ;
				%Curve Lines [id:da8666704460136332] 
				\draw [color={rgb, 255:red, 74; green, 144; blue, 226 }  ,draw opacity=1 ] [dash pattern={on 4.5pt off 4.5pt}]  (181.87,86.27) .. controls (183.2,97.6) and (213.2,110.27) .. (222.53,89.6) ;
				%Curve Lines [id:da3869227617197841] 
				\draw [color={rgb, 255:red, 74; green, 144; blue, 226 }  ,draw opacity=1 ]   (257.87,84.93) .. controls (253.1,60.96) and (307.1,63.46) .. (303.2,85.6) ;
				\draw [shift={(273.96,68.27)}, rotate = 356.96] [fill={rgb, 255:red, 74; green, 144; blue, 226 }  ,fill opacity=1 ][line width=0.08]  [draw opacity=0] (10.72,-5.15) -- (0,0) -- (10.72,5.15) -- (7.12,0) -- cycle    ;
				%Curve Lines [id:da40748578690315074] 
				\draw [color={rgb, 255:red, 74; green, 144; blue, 226 }  ,draw opacity=1 ] [dash pattern={on 4.5pt off 4.5pt}]  (257.87,84.93) .. controls (259.2,96.27) and (293.87,106.27) .. (303.2,85.6) ;
				
				% Text Node
				\draw (154,107.5) node [anchor=north west][inner sep=0.75pt]  [font=\small,color={rgb, 255:red, 208; green, 2; blue, 27 }  ,opacity=1 ] [align=left] {$\displaystyle a_{1}$};
				% Text Node
				\draw (233.83,106.5) node [anchor=north west][inner sep=0.75pt]  [font=\small,color={rgb, 255:red, 208; green, 2; blue, 27 }  ,opacity=1 ] [align=left] {$\displaystyle a_{2}$};
				% Text Node
				\draw (116.67,98) node [anchor=north west][inner sep=0.75pt]  [font=\small,color={rgb, 255:red, 74; green, 144; blue, 226 }  ,opacity=1 ] [align=left] {$\displaystyle b_{1}$};
				% Text Node
				\draw (194,101.33) node [anchor=north west][inner sep=0.75pt]  [font=\small,color={rgb, 255:red, 74; green, 144; blue, 226 }  ,opacity=1 ] [align=left] {$\displaystyle b_{2}$};
				% Text Node
				\draw (276.67,98.67) node [anchor=north west][inner sep=0.75pt]  [font=\small,color={rgb, 255:red, 74; green, 144; blue, 226 }  ,opacity=1 ] [align=left] {$\displaystyle b_{3}$};

			\end{tikzpicture}
			
			\caption{$a_i$ and $b_i$ on $\cH_2$}
			\label{fig: ab loops}

\end{figure}

\begin{proof}[Proof of Theorem \ref{thm: pentagon relation}]
	First, note that there is an $R$-\emph{module} isomorphism:
	$$
	\Sk(T-D) \longrightarrow \Sk (\cH_2), 
	$$
	given by the following construction. 
	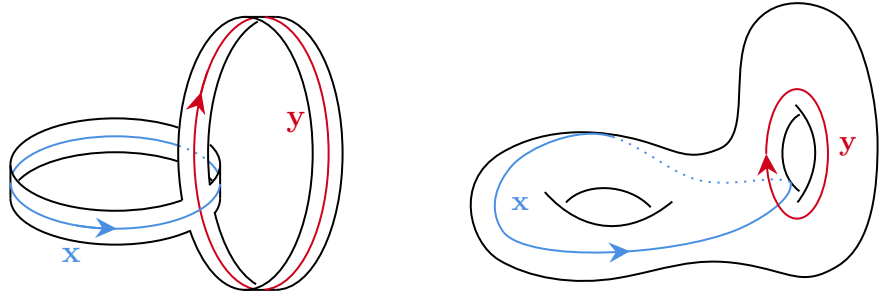
\begin{figure}[htbp]
		\centering

	\tikzset{every picture/.style={line width=0.75pt}} %set default line width to 0.75pt        
	
	\begin{tikzpicture}[x=0.75pt,y=0.75pt,yscale=-1,xscale=1]
		
		%Shape: Arc [id:dp9068236764994368] 
		\draw  [draw opacity=0][dash pattern={on 0.84pt off 2.51pt}] (174.47,136.14) .. controls (180.36,138.11) and (185.28,140.63) .. (188.85,143.51) -- (142.63,154.61) -- cycle ; \draw [color={rgb, 255:red, 74; green, 144; blue, 226 }  ,draw opacity=1 ][dash pattern={on 0.84pt off 2.51pt}] [dash pattern={on 0.84pt off 2.51pt}]  (174.47,136.14) .. controls (180.36,138.11) and (185.28,140.63) .. (188.85,143.51) ;  
		%Shape: Arc [id:dp8374015921868463] 
		\draw  [draw opacity=0] (91.17,149.74) .. controls (96.24,139.27) and (117.36,131.42) .. (142.63,131.42) .. controls (154.31,131.42) and (165.09,133.09) .. (173.82,135.93) -- (142.63,154.61) -- cycle ; \draw [color={rgb, 255:red, 74; green, 144; blue, 226 }  ,draw opacity=1 ]   (91.17,149.74) .. controls (96.24,139.27) and (117.36,131.42) .. (142.63,131.42) .. controls (154.31,131.42) and (165.09,133.09) .. (173.82,135.93) ;  
		%Shape: Arc [id:dp5758886249025319] 
		\draw  [draw opacity=0] (185.27,111.43) .. controls (190.05,90.31) and (199.97,74.97) .. (211.78,72.02) -- (215.95,140) -- cycle ; \draw [color={rgb, 255:red, 208; green, 2; blue, 27 }  ,draw opacity=1 ]   (185.5,110.44) .. controls (190.36,89.83) and (200.16,74.92) .. (211.78,72.02) ;  \draw [shift={(185.79,109.23)}, rotate = 106.02] [fill={rgb, 255:red, 208; green, 2; blue, 27 }  ,fill opacity=1 ][line width=0.08]  [draw opacity=0] (10.72,-5.15) -- (0,0) -- (10.72,5.15) -- (7.12,0) -- cycle    ;
		%Shape: Arc [id:dp7482085361478796] 
		\draw  [draw opacity=0] (211.31,207.86) .. controls (194.87,203.27) and (182.2,174.64) .. (182.2,140) .. controls (182.2,124.54) and (184.72,110.28) .. (188.98,98.81) -- (215.95,140) -- cycle ; \draw  [color={rgb, 255:red, 208; green, 2; blue, 27 }  ,draw opacity=1 ] (211.31,207.86) .. controls (194.87,203.27) and (182.2,174.64) .. (182.2,140) .. controls (182.2,124.54) and (184.72,110.28) .. (188.98,98.81) ;  
		%Shape: Arc [id:dp5664720118102373] 
		\draw  [draw opacity=0] (189.32,143.9) .. controls (191.47,145.71) and (193.09,147.67) .. (194.09,149.73) -- (142.63,154.61) -- cycle ; \draw [color={rgb, 255:red, 74; green, 144; blue, 226 }  ,draw opacity=1 ]   (189.32,143.9) .. controls (191.47,145.71) and (193.09,147.67) .. (194.09,149.73) ;  
		%Shape: Arc [id:dp7556520366697599] 
		\draw  [draw opacity=0] (176.23,163.46) .. controls (174.92,156.14) and (174.2,148.24) .. (174.2,140) .. controls (174.2,102.17) and (189.31,71.5) .. (207.95,71.5) .. controls (226.59,71.5) and (241.7,102.17) .. (241.7,140) .. controls (241.7,177.83) and (226.59,208.5) .. (207.95,208.5) .. controls (196.34,208.5) and (186.1,196.6) .. (180.03,178.49) -- (207.95,140) -- cycle ; \draw   (176.23,163.46) .. controls (174.92,156.14) and (174.2,148.24) .. (174.2,140) .. controls (174.2,102.17) and (189.31,71.5) .. (207.95,71.5) .. controls (226.59,71.5) and (241.7,102.17) .. (241.7,140) .. controls (241.7,177.83) and (226.59,208.5) .. (207.95,208.5) .. controls (196.34,208.5) and (186.1,196.6) .. (180.03,178.49) ;  
		%Shape: Arc [id:dp39912500528407935] 
		\draw  [draw opacity=0] (176.23,163.46) .. controls (167.12,166.79) and (155.41,168.8) .. (142.63,168.8) .. controls (113.57,168.8) and (90.01,158.41) .. (90.01,145.61) .. controls (90.01,132.8) and (113.57,122.42) .. (142.63,122.42) .. controls (154.48,122.42) and (165.41,124.14) .. (174.21,127.05) -- (142.63,145.61) -- cycle ; \draw   (176.23,163.46) .. controls (167.12,166.79) and (155.41,168.8) .. (142.63,168.8) .. controls (113.57,168.8) and (90.01,158.41) .. (90.01,145.61) .. controls (90.01,132.8) and (113.57,122.42) .. (142.63,122.42) .. controls (154.48,122.42) and (165.41,124.14) .. (174.21,127.05) ;  
		%Shape: Arc [id:dp3978956651467993] 
		\draw  [draw opacity=0] (222.95,71.5) .. controls (222.95,71.5) and (222.95,71.5) .. (222.95,71.5) .. controls (241.59,71.5) and (256.7,102.17) .. (256.7,140) .. controls (256.7,177.83) and (241.59,208.5) .. (222.95,208.5) -- (222.95,140) -- cycle ; \draw   (222.95,71.5) .. controls (222.95,71.5) and (222.95,71.5) .. (222.95,71.5) .. controls (241.59,71.5) and (256.7,102.17) .. (256.7,140) .. controls (256.7,177.83) and (241.59,208.5) .. (222.95,208.5) ;  
		%Shape: Arc [id:dp41407721207637027] 
		\draw  [draw opacity=0] (190.04,155.26) .. controls (189.49,150.35) and (189.2,145.24) .. (189.2,140) .. controls (189.2,108.34) and (199.78,81.7) .. (214.16,73.85) -- (222.95,140) -- cycle ; \draw   (190.04,155.26) .. controls (189.49,150.35) and (189.2,145.24) .. (189.2,140) .. controls (189.2,108.34) and (199.78,81.7) .. (214.16,73.85) ;  
		%Shape: Arc [id:dp5537142167390501] 
		\draw  [draw opacity=0] (190.13,135.61) .. controls (193.42,138.64) and (195.25,142.03) .. (195.25,145.61) .. controls (195.25,149.32) and (193.27,152.83) .. (189.75,155.95) -- (142.63,145.61) -- cycle ; \draw   (190.13,135.61) .. controls (193.42,138.64) and (195.25,142.03) .. (195.25,145.61) .. controls (195.25,149.32) and (193.27,152.83) .. (189.75,155.95) ;  
		%Shape: Arc [id:dp6561431984546078] 
		\draw  [draw opacity=0] (213.29,205.65) .. controls (204.08,200.08) and (196.52,186.73) .. (192.45,169.36) -- (222.95,140) -- cycle ; \draw   (213.29,205.65) .. controls (204.08,200.08) and (196.52,186.73) .. (192.45,169.36) ;  
		%Straight Lines [id:da5749997701234255] 
		\draw    (90.01,146) -- (90.01,162.61) ;
		%Shape: Arc [id:dp7118873013897244] 
		\draw  [draw opacity=0] (179.91,178.98) .. controls (170.38,183.19) and (157.2,185.8) .. (142.63,185.8) .. controls (113.57,185.8) and (90.01,175.41) .. (90.01,162.61) -- (142.63,162.61) -- cycle ; \draw   (179.91,178.98) .. controls (170.38,183.19) and (157.2,185.8) .. (142.63,185.8) .. controls (113.57,185.8) and (90.01,175.41) .. (90.01,162.61) ;  
		%Shape: Arc [id:dp40480306869027394] 
		\draw  [draw opacity=0] (195.25,162.61) .. controls (195.25,162.61) and (195.25,162.61) .. (195.25,162.61) .. controls (195.25,162.61) and (195.25,162.61) .. (195.25,162.61) .. controls (195.25,165.19) and (194.3,167.68) .. (192.52,170) -- (142.63,162.61) -- cycle ; \draw   (195.25,162.61) .. controls (195.25,162.61) and (195.25,162.61) .. (195.25,162.61) .. controls (195.25,162.61) and (195.25,162.61) .. (195.25,162.61) .. controls (195.25,165.19) and (194.3,167.68) .. (192.52,170) ;  
		%Straight Lines [id:da29931922128767785] 
		\draw    (195.25,146) -- (195.25,162.61) ;
		%Shape: Arc [id:dp3368284093366243] 
		\draw  [draw opacity=0] (94.13,153.59) .. controls (102.12,145.26) and (120.83,139.42) .. (142.63,139.42) .. controls (154.35,139.42) and (165.17,141.1) .. (173.92,143.96) -- (142.63,162.61) -- cycle ; \draw   (94.13,153.59) .. controls (102.12,145.26) and (120.83,139.42) .. (142.63,139.42) .. controls (154.35,139.42) and (165.17,141.1) .. (173.92,143.96) ;  
		%Shape: Arc [id:dp1895872550422466] 
		\draw  [draw opacity=0] (189.85,152.35) .. controls (190.35,152.81) and (190.83,153.27) .. (191.27,153.73) -- (142.63,162.61) -- cycle ; \draw   (189.85,152.35) .. controls (190.35,152.81) and (190.83,153.27) .. (191.27,153.73) ;  
		%Shape: Arc [id:dp33131041488804835] 
		\draw  [draw opacity=0] (143.33,177.79) .. controls (143.1,177.8) and (142.87,177.8) .. (142.63,177.8) .. controls (113.57,177.8) and (90.01,167.41) .. (90.01,154.61) -- (142.63,154.61) -- cycle ; \draw [color={rgb, 255:red, 74; green, 144; blue, 226 }  ,draw opacity=1 ]   (142.63,177.8) .. controls (113.57,177.8) and (90.01,167.41) .. (90.01,154.61) ;  \draw [shift={(143.33,177.79)}, rotate = 182.47] [fill={rgb, 255:red, 74; green, 144; blue, 226 }  ,fill opacity=1 ][line width=0.08]  [draw opacity=0] (10.72,-5.15) -- (0,0) -- (10.72,5.15) -- (7.12,0) -- cycle    ;
		%Shape: Arc [id:dp6572038134071514] 
		\draw  [draw opacity=0] (195.23,155.37) .. controls (194.31,167.82) and (171.12,177.8) .. (142.63,177.8) .. controls (135.7,177.8) and (129.07,177.21) .. (123.01,176.13) -- (142.63,154.61) -- cycle ; \draw [color={rgb, 255:red, 74; green, 144; blue, 226 }  ,draw opacity=1 ]   (195.23,155.37) .. controls (194.31,167.82) and (171.12,177.8) .. (142.63,177.8) .. controls (135.7,177.8) and (129.07,177.21) .. (123.01,176.13) ;  
		%Straight Lines [id:da3799707394875331] 
		\draw    (207,71.5) -- (222.95,71.5) ;
		%Straight Lines [id:da7593521796178018] 
		\draw    (207.22,208.5) -- (222.95,208.5) ;
		%Shape: Polygon Curved [id:ds9283383701325045] 
		\draw  [line width=0.75]  (336,140) .. controls (359.8,125.5) and (388.2,128.5) .. (406.2,133.5) .. controls (424.2,138.5) and (430.2,144.5) .. (446.2,137.5) .. controls (462.2,130.5) and (450.2,94.5) .. (462.2,76.5) .. controls (474.2,58.5) and (499.2,63.5) .. (509.2,74.5) .. controls (519.2,85.5) and (525.2,115.5) .. (525.2,135.5) .. controls (525.2,155.5) and (523.2,185.5) .. (505.2,196.5) .. controls (487.2,207.5) and (482.2,199.5) .. (457.2,195.5) .. controls (432.2,191.5) and (418.2,198.5) .. (395.2,202.5) .. controls (372.2,206.5) and (349.95,202.37) .. (335.2,193.5) .. controls (320.45,184.63) and (312.2,154.5) .. (336,140) -- cycle ;
		%Curve Lines [id:da4212895572115747] 
		\draw [line width=0.75]    (358.2,159.4) .. controls (377.89,181.33) and (401.52,181.33) .. (422.2,164.17) ;
		%Curve Lines [id:da1805896350513818] 
		\draw [line width=0.75]    (368.83,164.65) .. controls (377.89,154.64) and (399.55,154.64) .. (411.37,167.03) ;
		%Curve Lines [id:da8293815337542967] 
		\draw [line width=0.75]    (484,115.6) .. controls (499,131.6) and (495,151.6) .. (485,164.6) ;
		%Curve Lines [id:da3075753100356031] 
		\draw [line width=0.75]    (486.2,120.5) .. controls (477.2,124.5) and (472.2,148.5) .. (486,159) ;
		%Curve Lines [id:da4720960716733551] 
		\draw [color={rgb, 255:red, 74; green, 144; blue, 226 }  ,draw opacity=1 ][line width=0.75]    (337.53,179.67) .. controls (350.87,195) and (422.2,191.5) .. (454.2,178.5) ;
		\draw [shift={(400.99,189.11)}, rotate = 176.53] [fill={rgb, 255:red, 74; green, 144; blue, 226 }  ,fill opacity=1 ][line width=0.08]  [draw opacity=0] (10.72,-5.15) -- (0,0) -- (10.72,5.15) -- (7.12,0) -- cycle    ;
		%Curve Lines [id:da051082400353493096] 
		\draw [color={rgb, 255:red, 74; green, 144; blue, 226 }  ,draw opacity=1 ][line width=0.75]    (454.2,178.5) .. controls (464.2,174.5) and (483.53,163.33) .. (481.53,154.33) ;
		%Curve Lines [id:da13345519820508378] 
		\draw [color={rgb, 255:red, 74; green, 144; blue, 226 }  ,draw opacity=1 ][line width=0.75]  [dash pattern={on 0.84pt off 2.51pt}]  (392.2,131.67) .. controls (411.53,138.33) and (423.34,152.39) .. (440.2,154.33) .. controls (457.06,156.27) and (471.84,150.61) .. (481.53,154.33) ;
		%Curve Lines [id:da7914063807696614] 
		\draw [color={rgb, 255:red, 74; green, 144; blue, 226 }  ,draw opacity=1 ][line width=0.75]    (337.53,179.67) .. controls (328.87,167.67) and (334.2,154.33) .. (342.2,145.67) .. controls (350.2,137) and (375.02,126.07) .. (392.2,131.67) ;
		%Shape: Arc [id:dp5884436956108625] 
		\draw  [draw opacity=0][line width=0.75]  (469.3,140.25) .. controls (469.3,140.25) and (469.3,140.25) .. (469.3,140.25) .. controls (469.3,122.3) and (476.24,107.75) .. (484.8,107.75) .. controls (493.36,107.75) and (500.3,122.3) .. (500.3,140.25) .. controls (500.3,158.2) and (493.36,172.75) .. (484.8,172.75) .. controls (476.24,172.75) and (469.3,158.2) .. (469.3,140.25) -- (484.8,140.25) -- cycle ; \draw [color={rgb, 255:red, 208; green, 2; blue, 27 }  ,draw opacity=1 ][line width=0.75]    (469.3,140.25) .. controls (469.3,122.3) and (476.24,107.75) .. (484.8,107.75) .. controls (493.36,107.75) and (500.3,122.3) .. (500.3,140.25) .. controls (500.3,158.2) and (493.36,172.75) .. (484.8,172.75) .. controls (476.71,172.75) and (470.07,159.76) .. (469.36,143.18) ; \draw [shift={(469.3,140.25)}, rotate = 86.44] [fill={rgb, 255:red, 208; green, 2; blue, 27 }  ,fill opacity=1 ][line width=0.08]  [draw opacity=0] (10.72,-5.15) -- (0,0) -- (10.72,5.15) -- (7.12,0) -- cycle    ; 
		%Shape: Arc [id:dp405949516570663] 
		\draw  [draw opacity=0] (217.68,71.59) .. controls (235.52,73.42) and (249.7,103.35) .. (249.7,140) .. controls (249.7,176.79) and (235.41,206.81) .. (217.48,208.43) -- (215.95,140) -- cycle ; \draw  [color={rgb, 255:red, 208; green, 2; blue, 27 }  ,draw opacity=1 ] (217.68,71.59) .. controls (235.52,73.42) and (249.7,103.35) .. (249.7,140) .. controls (249.7,176.79) and (235.41,206.81) .. (217.48,208.43) ;  
		
		% Text Node
		\draw (114.33,186) node [anchor=north west][inner sep=0.75pt]  [color={rgb, 255:red, 74; green, 144; blue, 226 }  ,opacity=1 ] [align=left] {$\displaystyle \mathbf{x}$};
		% Text Node
		\draw (227,117.67) node [anchor=north west][inner sep=0.75pt]  [color={rgb, 255:red, 208; green, 2; blue, 27 }  ,opacity=1 ] [align=left] {$\displaystyle \mathbf{y}$};
		% Text Node
		\draw (340,160.33) node [anchor=north west][inner sep=0.75pt]  [font=\small,color={rgb, 255:red, 74; green, 144; blue, 226 }  ,opacity=1 ] [align=left] {$\displaystyle \mathbf{x}$};
		% Text Node
		\draw (504.5,130.33) node [anchor=north west][inner sep=0.75pt]  [font=\small,color={rgb, 255:red, 208; green, 2; blue, 27 }  ,opacity=1 ] [align=left] {$\displaystyle \mathbf{y}$};

	\end{tikzpicture}
		
		\caption{$T-D$ (left) and $(T-D) \times I$ (right)}
		\label{fig: punctured torus}
	\end{figure}
	
	The left picture in Figure \ref{fig: punctured torus} is homeomorphic to  $T - D$, while the right is homeomorphic to $(T-D) \times I$, where $I = [0, 1]$. We take the convention that the loops $\x \times \{1\}$ and $\y \times \{1\}$ are drawn as shown in the picture. 
	
	%Now we can draw $P_\x \cdot P_\y$ and $P_\y \cdot P_{\x + \y} \cdot P_\x$ as in Figure \ref{subfig: skeins on T-D}.
	
	On the other hand, we can do a ``$\pi/4$ twist'' to identify $(T - D) \times I$ with the standard handlebody $\cH_2$, as in Figure \ref{fig: twist}.  This homeomorphism would send the loop $\x \times \{1\}$ to $b_1 + a_1 + b_2$, and $\y$ to $a_2$, where $a_i$ and $b_i$ are some standard loops on $\Sigma_2 = \partial \cH_2$, drawn in Figure \ref{fig: ab loops}, and ``$+$'' means doing a surgery at the intersection point, as in Notation \ref{notation: x+l}.

	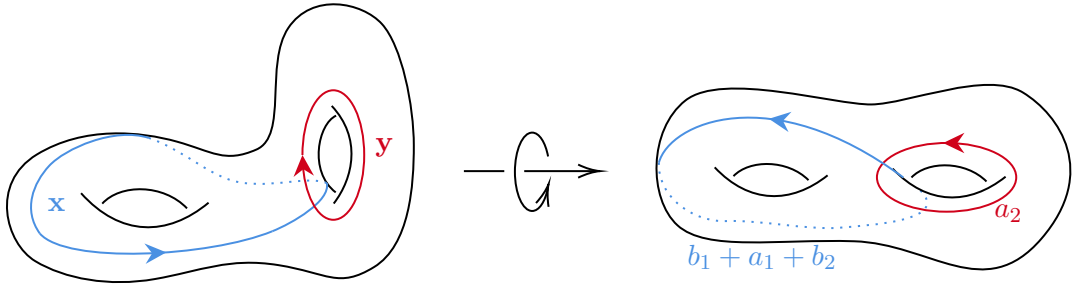
\begin{figure}[htbp]
		\centering

		\tikzset{every picture/.style={line width=0.75pt}} %set default line width to 0.75pt        
		
		\begin{tikzpicture}[x=0.75pt,y=0.75pt,yscale=-1,xscale=1]
			
			%Curve Lines [id:da1563206855123751] 
			\draw [color={rgb, 255:red, 74; green, 144; blue, 226 }  ,draw opacity=1 ] [dash pattern={on 0.84pt off 2.51pt}]  (393.2,127.07) .. controls (392.53,144.4) and (410.53,157.07) .. (427.2,156.4) .. controls (443.87,155.73) and (468.53,161.07) .. (483.87,159.73) .. controls (499.2,158.4) and (541.74,156.35) .. (523.2,139.87) ;
			%Shape: Polygon Curved [id:ds5935312088256339] 
			\draw  [line width=0.75]  (81,123) .. controls (104.8,108.5) and (133.2,111.5) .. (151.2,116.5) .. controls (169.2,121.5) and (175.2,127.5) .. (191.2,120.5) .. controls (207.2,113.5) and (195.2,77.5) .. (207.2,59.5) .. controls (219.2,41.5) and (244.2,46.5) .. (254.2,57.5) .. controls (264.2,68.5) and (270.2,98.5) .. (270.2,118.5) .. controls (270.2,138.5) and (268.2,168.5) .. (250.2,179.5) .. controls (232.2,190.5) and (227.2,182.5) .. (202.2,178.5) .. controls (177.2,174.5) and (163.2,181.5) .. (140.2,185.5) .. controls (117.2,189.5) and (94.95,185.37) .. (80.2,176.5) .. controls (65.45,167.63) and (57.2,137.5) .. (81,123) -- cycle ;
			%Curve Lines [id:da217994153810654] 
			\draw [line width=0.75]    (103.2,142.4) .. controls (122.89,164.33) and (146.52,164.33) .. (167.2,147.17) ;
			%Curve Lines [id:da018959718241639667] 
			\draw [line width=0.75]    (113.83,147.65) .. controls (122.89,137.64) and (144.55,137.64) .. (156.37,150.03) ;
			%Curve Lines [id:da8317643292146217] 
			\draw [line width=0.75]    (229,98.6) .. controls (244,114.6) and (240,134.6) .. (230,147.6) ;
			%Curve Lines [id:da18777810044113186] 
			\draw [line width=0.75]    (231.2,103.5) .. controls (222.2,107.5) and (217.2,131.5) .. (231,142) ;
			%Curve Lines [id:da3106563889812659] 
			\draw [color={rgb, 255:red, 74; green, 144; blue, 226 }  ,draw opacity=1 ][line width=0.75]    (82.53,162.67) .. controls (95.87,178) and (167.2,174.5) .. (199.2,161.5) ;
			\draw [shift={(145.99,172.11)}, rotate = 176.53] [fill={rgb, 255:red, 74; green, 144; blue, 226 }  ,fill opacity=1 ][line width=0.08]  [draw opacity=0] (10.72,-5.15) -- (0,0) -- (10.72,5.15) -- (7.12,0) -- cycle    ;
			%Curve Lines [id:da8352680999106077] 
			\draw [color={rgb, 255:red, 74; green, 144; blue, 226 }  ,draw opacity=1 ][line width=0.75]    (199.2,161.5) .. controls (209.2,157.5) and (228.53,146.33) .. (226.53,137.33) ;
			%Curve Lines [id:da29860589938868065] 
			\draw [color={rgb, 255:red, 74; green, 144; blue, 226 }  ,draw opacity=1 ][line width=0.75]  [dash pattern={on 0.84pt off 2.51pt}]  (137.2,114.67) .. controls (156.53,121.33) and (168.34,135.39) .. (185.2,137.33) .. controls (202.06,139.27) and (216.84,133.61) .. (226.53,137.33) ;
			%Curve Lines [id:da0018542529369869332] 
			\draw [color={rgb, 255:red, 74; green, 144; blue, 226 }  ,draw opacity=1 ][line width=0.75]    (82.53,162.67) .. controls (73.87,150.67) and (79.2,137.33) .. (87.2,128.67) .. controls (95.2,120) and (120.02,109.07) .. (137.2,114.67) ;
			%Shape: Arc [id:dp9936096306591933] 
			\draw  [draw opacity=0][line width=0.75]  (214.3,123.25) .. controls (214.3,123.25) and (214.3,123.25) .. (214.3,123.25) .. controls (214.3,105.3) and (221.24,90.75) .. (229.8,90.75) .. controls (238.36,90.75) and (245.3,105.3) .. (245.3,123.25) .. controls (245.3,141.2) and (238.36,155.75) .. (229.8,155.75) .. controls (221.24,155.75) and (214.3,141.2) .. (214.3,123.25) -- (229.8,123.25) -- cycle ; \draw [color={rgb, 255:red, 208; green, 2; blue, 27 }  ,draw opacity=1 ][line width=0.75]    (214.3,123.25) .. controls (214.3,105.3) and (221.24,90.75) .. (229.8,90.75) .. controls (238.36,90.75) and (245.3,105.3) .. (245.3,123.25) .. controls (245.3,141.2) and (238.36,155.75) .. (229.8,155.75) .. controls (221.71,155.75) and (215.07,142.76) .. (214.36,126.18) ; \draw [shift={(214.3,123.25)}, rotate = 86.44] [fill={rgb, 255:red, 208; green, 2; blue, 27 }  ,fill opacity=1 ][line width=0.08]  [draw opacity=0] (10.72,-5.15) -- (0,0) -- (10.72,5.15) -- (7.12,0) -- cycle    ; 
			%Straight Lines [id:da8054040291376556] 
			\draw    (295.33,131.33) -- (315.6,131.33) ;
			%Straight Lines [id:da12886033659021234] 
			\draw    (325.6,131.33) -- (361.6,131.33) ;
			\draw [shift={(363.6,131.33)}, rotate = 180] [color={rgb, 255:red, 0; green, 0; blue, 0 }  ][line width=0.75]    (10.93,-3.29) .. controls (6.95,-1.4) and (3.31,-0.3) .. (0,0) .. controls (3.31,0.3) and (6.95,1.4) .. (10.93,3.29)   ;
			%Shape: Arc [id:dp22595779991724552] 
			\draw  [draw opacity=0] (337.83,137.84) .. controls (336.69,145.68) and (333.44,151.33) .. (329.6,151.33) .. controls (324.81,151.33) and (320.93,142.53) .. (320.93,131.67) .. controls (320.93,120.81) and (324.81,112) .. (329.6,112) .. controls (333.18,112) and (336.26,116.93) .. (337.58,123.97) -- (329.6,131.67) -- cycle ; \draw    (337.47,139.9) .. controls (336.1,146.65) and (333.09,151.33) .. (329.6,151.33) .. controls (324.81,151.33) and (320.93,142.53) .. (320.93,131.67) .. controls (320.93,120.81) and (324.81,112) .. (329.6,112) .. controls (333.18,112) and (336.26,116.93) .. (337.58,123.97) ;  \draw [shift={(337.83,137.84)}, rotate = 106.78] [color={rgb, 255:red, 0; green, 0; blue, 0 }  ][line width=0.75]    (10.93,-3.29) .. controls (6.95,-1.4) and (3.31,-0.3) .. (0,0) .. controls (3.31,0.3) and (6.95,1.4) .. (10.93,3.29)   ;
			%Shape: Polygon Curved [id:ds023326439458080772] 
			\draw   (406.53,95.87) .. controls (425.2,81.2) and (481.38,98.13) .. (499.87,97.87) .. controls (518.35,97.61) and (557.47,80.66) .. (575.87,91.87) .. controls (594.27,103.07) and (613.87,131.87) .. (585.2,165.2) .. controls (556.53,198.53) and (529.64,167.74) .. (491.2,166.53) .. controls (452.76,165.33) and (417.2,171.87) .. (402.53,158.53) .. controls (387.87,145.2) and (387.87,110.53) .. (406.53,95.87) -- cycle ;
			%Curve Lines [id:da13602435334564977] 
			\draw [line width=0.75]    (421.2,129.43) .. controls (438.64,147.95) and (459.56,147.95) .. (477.87,133.45) ;
			%Curve Lines [id:da19756362477604195] 
			\draw [line width=0.75]    (430.62,133.86) .. controls (438.64,125.4) and (457.82,125.4) .. (468.28,135.87) ;
			%Curve Lines [id:da13695433178690197] 
			\draw [line width=0.75]    (510.53,130.09) .. controls (527.97,148.61) and (548.89,148.61) .. (567.2,134.12) ;
			%Curve Lines [id:da11172048896254494] 
			\draw [line width=0.75]    (519.95,134.52) .. controls (527.97,126.07) and (547.15,126.07) .. (557.61,136.54) ;
			%Shape: Arc [id:dp35221412157491416] 
			\draw  [draw opacity=0][line width=0.75]  (537.88,117.2) .. controls (557.04,117.28) and (572.53,124.97) .. (572.53,134.44) .. controls (572.53,143.97) and (556.88,151.69) .. (537.57,151.69) .. controls (518.27,151.69) and (502.62,143.97) .. (502.62,134.44) .. controls (502.62,124.92) and (518.27,117.2) .. (537.57,117.2) -- (537.57,134.44) -- cycle ; \draw [color={rgb, 255:red, 208; green, 2; blue, 27 }  ,draw opacity=1 ][line width=0.75]    (537.88,117.2) .. controls (557.04,117.28) and (572.53,124.97) .. (572.53,134.44) .. controls (572.53,143.97) and (556.88,151.69) .. (537.57,151.69) .. controls (518.27,151.69) and (502.62,143.97) .. (502.62,134.44) .. controls (502.62,124.74) and (518.85,116.92) .. (538.65,117.21) ; \draw [shift={(535.84,117.22)}, rotate = 360] [fill={rgb, 255:red, 208; green, 2; blue, 27 }  ,fill opacity=1 ][line width=0.08]  [draw opacity=0] (10.72,-5.15) -- (0,0) -- (10.72,5.15) -- (7.12,0) -- cycle    ; 
			%Curve Lines [id:da530605681825385] 
			\draw [color={rgb, 255:red, 74; green, 144; blue, 226 }  ,draw opacity=1 ][line width=0.75]    (515.87,134.53) .. controls (449.2,85.07) and (398.53,107.07) .. (393.2,127.07) ;
			\draw [shift={(448.41,105.07)}, rotate = 8.36] [fill={rgb, 255:red, 74; green, 144; blue, 226 }  ,fill opacity=1 ][line width=0.08]  [draw opacity=0] (10.72,-5.15) -- (0,0) -- (10.72,5.15) -- (7.12,0) -- cycle    ;
			
			% Text Node
			\draw (85,143.33) node [anchor=north west][inner sep=0.75pt]  [font=\small,color={rgb, 255:red, 74; green, 144; blue, 226 }  ,opacity=1 ] [align=left] {$\displaystyle \mathbf{x}$};
			% Text Node
			\draw (249.5,113.33) node [anchor=north west][inner sep=0.75pt]  [font=\small,color={rgb, 255:red, 208; green, 2; blue, 27 }  ,opacity=1 ] [align=left] {$\displaystyle \mathbf{y}$};
			% Text Node
			\draw (406,166.6) node [anchor=north west][inner sep=0.75pt]  [font=\small,color={rgb, 255:red, 74; green, 144; blue, 226 }  ,opacity=1 ] [align=left] {$\displaystyle b_{1} +a_{1} +b_2$};
			% Text Node
			\draw (560.17,146.67) node [anchor=north west][inner sep=0.75pt]  [font=\small,color={rgb, 255:red, 208; green, 2; blue, 27 }  ,opacity=1 ] [align=left] {$\displaystyle a_{2}$};

		\end{tikzpicture}
		\caption{The homeomorphism between $(T-D) \times I$ and $\cH_2$}
		\label{fig: twist}
	\end{figure}
	
	The induced skein \emph{module} isomorphism:
	$$
	\Sk( (T-D) \times I) \longrightarrow \Sk (\cH_2),
	$$
	will send 
	\begin{align*}
		P_\x \cdot P_\y  & \mapsto P_{b_1 + a_1 + b_2} \cdot P_{a_2} \cdot [\emptyset] , \\ 
		P_{\y} \cdot P_{\x + \y} \cdot P_\x & \mapsto P_{a_2} \cdot P_{b_1 + a_1 +b_2 + a_2} \cdot P_{b_1 + a_1 + b_2} \cdot [\emptyset],
	\end{align*}
	as shown in Figure \ref{fig: loops}.
	Here the right hand side should be viewed as elements in $\Sk(\Sigma_2)$ acting on $\cH_2$, and $[\emptyset]$ means the skein class given by the empty link. Actually, for an element $P_l \in \Sk(\Sigma_2)$, $P_l \cdot [\emptyset]$ is just its image in $\Sk(\cH_2$) under the inclusion map.
	
	\begin{figure}[htbp]
		\begin{subfigure}{1 \textwidth} %xy on T-D 
			\centering

		\tikzset{every picture/.style={line width=0.75pt}} %set default line width to 0.75pt        
		
		\begin{tikzpicture}[x=0.75pt,y=0.75pt,yscale=-1,xscale=1]
			
			%Curve Lines [id:da11363443924458383] 
			\draw [color={rgb, 255:red, 144; green, 19; blue, 254 }  ,draw opacity=1 ] [dash pattern={on 0.84pt off 2.51pt}]  (363.87,127.33) .. controls (386.46,110.39) and (411.01,118.32) .. (419.87,122.13) .. controls (428.73,125.95) and (456.53,139.47) .. (473.2,141.47) .. controls (489.87,143.47) and (511.2,130.13) .. (517.2,138.8) ;
			%Shape: Polygon Curved [id:ds8149965838801316] 
			\draw  [line width=0.75]  (113,123) .. controls (136.8,108.5) and (165.2,111.5) .. (183.2,116.5) .. controls (201.2,121.5) and (207.2,127.5) .. (223.2,120.5) .. controls (239.2,113.5) and (227.2,77.5) .. (239.2,59.5) .. controls (251.2,41.5) and (276.2,46.5) .. (286.2,57.5) .. controls (296.2,68.5) and (302.2,98.5) .. (302.2,118.5) .. controls (302.2,138.5) and (300.2,168.5) .. (282.2,179.5) .. controls (264.2,190.5) and (259.2,182.5) .. (234.2,178.5) .. controls (209.2,174.5) and (195.2,181.5) .. (172.2,185.5) .. controls (149.2,189.5) and (126.95,185.37) .. (112.2,176.5) .. controls (97.45,167.63) and (89.2,137.5) .. (113,123) -- cycle ;
			%Curve Lines [id:da23087599983740548] 
			\draw [line width=0.75]    (135.2,142.4) .. controls (154.89,164.33) and (178.52,164.33) .. (199.2,147.17) ;
			%Curve Lines [id:da6973289273387526] 
			\draw [line width=0.75]    (145.83,147.65) .. controls (154.89,137.64) and (176.55,137.64) .. (188.37,150.03) ;
			%Curve Lines [id:da5334644006135709] 
			\draw [line width=0.75]    (258.2,90.5) .. controls (278.2,106.5) and (272.2,141.5) .. (259.2,153.5) ;
			%Curve Lines [id:da302359064252574] 
			\draw [line width=0.75]    (263.2,103.5) .. controls (254.2,107.5) and (249.2,131.5) .. (263,142) ;
			%Curve Lines [id:da9511473691055121] 
			\draw [color={rgb, 255:red, 74; green, 144; blue, 226 }  ,draw opacity=1 ][line width=0.75]    (114.53,162.67) .. controls (127.87,178) and (199.2,174.5) .. (231.2,161.5) ;
			\draw [shift={(177.99,172.11)}, rotate = 176.53] [fill={rgb, 255:red, 74; green, 144; blue, 226 }  ,fill opacity=1 ][line width=0.08]  [draw opacity=0] (10.72,-5.15) -- (0,0) -- (10.72,5.15) -- (7.12,0) -- cycle    ;
			%Curve Lines [id:da94588921596428] 
			\draw [color={rgb, 255:red, 74; green, 144; blue, 226 }  ,draw opacity=1 ][line width=0.75]    (231.2,161.5) .. controls (241.2,157.5) and (260.53,146.33) .. (258.53,137.33) ;
			%Curve Lines [id:da6348332348394199] 
			\draw [color={rgb, 255:red, 74; green, 144; blue, 226 }  ,draw opacity=1 ][line width=0.75]  [dash pattern={on 0.84pt off 2.51pt}]  (169.2,114.67) .. controls (188.53,121.33) and (200.34,135.39) .. (217.2,137.33) .. controls (234.06,139.27) and (248.84,133.61) .. (258.53,137.33) ;
			%Curve Lines [id:da4061828272549055] 
			\draw [color={rgb, 255:red, 74; green, 144; blue, 226 }  ,draw opacity=1 ][line width=0.75]    (114.53,162.67) .. controls (105.87,150.67) and (111.2,137.33) .. (119.2,128.67) .. controls (127.2,120) and (152.02,109.07) .. (169.2,114.67) ;
			%Shape: Arc [id:dp22253022528107436] 
			\draw  [draw opacity=0][line width=0.75]  (243.91,144.36) .. controls (242.19,138.09) and (241.2,130.79) .. (241.2,123) .. controls (241.2,99.53) and (250.15,80.5) .. (261.2,80.5) .. controls (272.25,80.5) and (281.2,99.53) .. (281.2,123) .. controls (281.2,146.47) and (272.25,165.5) .. (261.2,165.5) .. controls (257.21,165.5) and (253.49,163.02) .. (250.37,158.74) -- (261.2,123) -- cycle ; \draw [color={rgb, 255:red, 208; green, 2; blue, 27 }  ,draw opacity=1 ][line width=0.75]    (243.91,144.36) .. controls (242.19,138.09) and (241.2,130.79) .. (241.2,123) .. controls (241.2,99.53) and (250.15,80.5) .. (261.2,80.5) .. controls (272.25,80.5) and (281.2,99.53) .. (281.2,123) .. controls (281.2,146.47) and (272.25,165.5) .. (261.2,165.5) .. controls (257.99,165.5) and (254.95,163.89) .. (252.26,161.03) ; \draw [shift={(250.37,158.74)}, rotate = 32.71] [fill={rgb, 255:red, 208; green, 2; blue, 27 }  ,fill opacity=1 ][line width=0.08]  [draw opacity=0] (10.72,-5.15) -- (0,0) -- (10.72,5.15) -- (7.12,0) -- cycle    ; 
			%Shape: Polygon Curved [id:ds42323109307959905] 
			\draw  [line width=0.75]  (369.67,121.67) .. controls (393.47,107.17) and (421.87,110.17) .. (439.87,115.17) .. controls (457.87,120.17) and (463.87,126.17) .. (479.87,119.17) .. controls (495.87,112.17) and (483.87,76.17) .. (495.87,58.17) .. controls (507.87,40.17) and (532.87,45.17) .. (542.87,56.17) .. controls (552.87,67.17) and (558.87,97.17) .. (558.87,117.17) .. controls (558.87,137.17) and (556.87,167.17) .. (538.87,178.17) .. controls (520.87,189.17) and (515.87,181.17) .. (490.87,177.17) .. controls (465.87,173.17) and (451.87,180.17) .. (428.87,184.17) .. controls (405.87,188.17) and (383.62,184.03) .. (368.87,175.17) .. controls (354.11,166.3) and (345.87,136.17) .. (369.67,121.67) -- cycle ;
			%Curve Lines [id:da8095585748906122] 
			\draw [line width=0.75]    (391.87,141.07) .. controls (411.56,163) and (435.19,163) .. (455.87,145.84) ;
			%Curve Lines [id:da6919819201613091] 
			\draw [line width=0.75]    (402.5,146.31) .. controls (411.56,136.3) and (433.22,136.3) .. (445.04,148.7) ;
			%Curve Lines [id:da3241651267285468] 
			\draw [line width=0.75]    (514.87,89.17) .. controls (534.87,105.17) and (531.87,139.33) .. (522.53,152.67) ;
			%Curve Lines [id:da5260441775669966] 
			\draw [line width=0.75]    (519.87,102.17) .. controls (510.87,106.17) and (505.87,134) .. (523.2,143.33) ;
			%Curve Lines [id:da40810200394410767] 
			\draw [color={rgb, 255:red, 74; green, 144; blue, 226 }  ,draw opacity=1 ][line width=0.75]    (386.53,164.67) .. controls (403.87,176.67) and (475.2,170.67) .. (499.2,150) ;
			\draw [shift={(449.35,168.37)}, rotate = 171.81] [fill={rgb, 255:red, 74; green, 144; blue, 226 }  ,fill opacity=1 ][line width=0.08]  [draw opacity=0] (10.72,-5.15) -- (0,0) -- (10.72,5.15) -- (7.12,0) -- cycle    ;
			%Curve Lines [id:da023339857890086524] 
			\draw [color={rgb, 255:red, 74; green, 144; blue, 226 }  ,draw opacity=1 ][line width=0.75]    (505.87,144.67) .. controls (512.53,141.33) and (512.53,128) .. (509.87,122.67) ;
			%Curve Lines [id:da3654499277741936] 
			\draw [color={rgb, 255:red, 74; green, 144; blue, 226 }  ,draw opacity=1 ][line width=0.75]  [dash pattern={on 0.84pt off 2.51pt}]  (425.87,113.33) .. controls (445.2,120) and (457,134.06) .. (473.87,136) .. controls (490.73,137.94) and (500.18,118.94) .. (509.87,122.67) ;
			%Curve Lines [id:da2619172462629511] 
			\draw [color={rgb, 255:red, 74; green, 144; blue, 226 }  ,draw opacity=1 ][line width=0.75]    (376.53,158.67) .. controls (367.2,152) and (367.87,136) .. (375.87,127.33) .. controls (383.87,118.67) and (408.68,107.74) .. (425.87,113.33) ;
			%Shape: Arc [id:dp5941435725451785] 
			\draw  [draw opacity=0][line width=0.75]  (537.87,121.67) .. controls (537.87,145.14) and (528.91,164.17) .. (517.87,164.17) .. controls (506.82,164.17) and (497.87,145.14) .. (497.87,121.67) .. controls (497.87,98.19) and (506.82,79.17) .. (517.87,79.17) .. controls (528.86,79.17) and (537.78,98.01) .. (537.87,121.32) -- (517.87,121.67) -- cycle ; \draw [color={rgb, 255:red, 208; green, 2; blue, 27 }  ,draw opacity=1 ][line width=0.75]    (537.87,121.67) .. controls (537.87,145.14) and (528.91,164.17) .. (517.87,164.17) .. controls (506.82,164.17) and (497.87,145.14) .. (497.87,121.67) .. controls (497.87,98.19) and (506.82,79.17) .. (517.87,79.17) .. controls (528.42,79.17) and (537.06,96.53) .. (537.81,118.54) ; \draw [shift={(537.87,121.32)}, rotate = 261.83] [fill={rgb, 255:red, 208; green, 2; blue, 27 }  ,fill opacity=1 ][line width=0.08]  [draw opacity=0] (10.72,-5.15) -- (0,0) -- (10.72,5.15) -- (7.12,0) -- cycle    ; 
			%Curve Lines [id:da3631111244973868] 
			\draw [color={rgb, 255:red, 144; green, 19; blue, 254 }  ,draw opacity=1 ][line width=0.75]    (363.87,127.33) .. controls (354.53,138.67) and (364.53,168.67) .. (380.53,162) .. controls (396.53,155.33) and (410.61,160.19) .. (421.9,162.47) .. controls (433.19,164.75) and (475.87,158) .. (484.53,144) .. controls (493.2,130) and (489.95,106.42) .. (497.87,83.33) .. controls (505.78,60.24) and (522.53,61.33) .. (529.87,66) .. controls (537.2,70.67) and (548.53,85.33) .. (551.2,110) .. controls (553.87,134.67) and (542.53,165.47) .. (529.2,162.13) ;
			%Curve Lines [id:da6293065798566655] 
			\draw [color={rgb, 255:red, 144; green, 19; blue, 254 }  ,draw opacity=1 ]   (517.2,138.8) .. controls (519.87,146.13) and (513.87,152.13) .. (522.53,158.8) ;
			
			% Text Node
			\draw (117,143.33) node [anchor=north west][inner sep=0.75pt]  [font=\small,color={rgb, 255:red, 74; green, 144; blue, 226 }  ,opacity=1 ] [align=left] {$\displaystyle \mathbf{x}$};
			% Text Node
			\draw (285,113.33) node [anchor=north west][inner sep=0.75pt]  [font=\small,color={rgb, 255:red, 208; green, 2; blue, 27 }  ,opacity=1 ] [align=left] {$\displaystyle \mathbf{y}$};
			% Text Node
			\draw (490.33,156.67) node [anchor=north west][inner sep=0.75pt]  [font=\small,color={rgb, 255:red, 74; green, 144; blue, 226 }  ,opacity=1 ] [align=left] {$\displaystyle \mathbf{x}$};
			% Text Node
			\draw (518.33,164) node [anchor=north west][inner sep=0.75pt]  [font=\small,color={rgb, 255:red, 208; green, 2; blue, 27 }  ,opacity=1 ] [align=left] {$\displaystyle \mathbf{y}$};
			% Text Node
			\draw (565,108.67) node [anchor=north west][inner sep=0.75pt]  [font=\small,color={rgb, 255:red, 144; green, 19; blue, 254 }  ,opacity=1 ] [align=left] {$\displaystyle \mathbf{x} +\mathbf{y}$};

		\end{tikzpicture}
			\caption{$P_{\x} \cdot P_\y$ (left) and $P_\y \cdot P_{\x+\y} \cdot P_\x $ (right) on $(T-D) \times I$}
			\label{subfig: skeins on T-D}
		\end{subfigure}
		\begin{subfigure}{1 \textwidth} %xy on H_2
			\centering

		\tikzset{every picture/.style={line width=0.75pt}} %set default line width to 0.75pt        
		
		\begin{tikzpicture}[x=0.75pt,y=0.75pt,yscale=-1,xscale=1]
			
			%Curve Lines [id:da4855281736578536] 
			\draw [color={rgb, 255:red, 74; green, 144; blue, 226 }  ,draw opacity=1 ] [dash pattern={on 0.84pt off 2.51pt}]  (67.2,115.07) .. controls (66.53,132.4) and (84.53,145.07) .. (101.2,144.4) .. controls (117.87,143.73) and (142.53,149.07) .. (157.87,147.73) .. controls (173.2,146.4) and (215.74,144.35) .. (197.2,127.87) ;
			%Shape: Polygon Curved [id:ds01224191691688592] 
			\draw   (80.53,83.87) .. controls (99.2,69.2) and (155.38,86.13) .. (173.87,85.87) .. controls (192.35,85.61) and (231.47,68.66) .. (249.87,79.87) .. controls (268.27,91.07) and (287.87,119.87) .. (259.2,153.2) .. controls (230.53,186.53) and (203.64,155.74) .. (165.2,154.53) .. controls (126.76,153.33) and (91.2,159.87) .. (76.53,146.53) .. controls (61.87,133.2) and (61.87,98.53) .. (80.53,83.87) -- cycle ;
			%Curve Lines [id:da5824043650383581] 
			\draw [line width=0.75]    (95.2,117.43) .. controls (112.64,135.95) and (133.56,135.95) .. (151.87,121.45) ;
			%Curve Lines [id:da4514266129895106] 
			\draw [line width=0.75]    (104.62,121.86) .. controls (112.64,113.4) and (131.82,113.4) .. (142.28,123.87) ;
			%Curve Lines [id:da6207104290048284] 
			\draw [line width=0.75]    (184.53,118.09) .. controls (201.97,136.61) and (222.89,136.61) .. (241.2,122.12) ;
			%Curve Lines [id:da22622702485412582] 
			\draw [line width=0.75]    (193.95,122.52) .. controls (201.97,114.07) and (221.15,114.07) .. (231.61,124.54) ;
			%Shape: Arc [id:dp06779011633990506] 
			\draw  [draw opacity=0][line width=0.75]  (181.63,113.58) .. controls (187.75,108.59) and (198.87,105.25) .. (211.57,105.25) .. controls (230.88,105.25) and (246.53,112.96) .. (246.53,122.47) .. controls (246.53,131.98) and (230.88,139.69) .. (211.57,139.69) .. controls (192.27,139.69) and (176.62,131.98) .. (176.62,122.47) .. controls (176.62,120.99) and (176.99,119.55) .. (177.71,118.18) -- (211.57,122.47) -- cycle ; \draw [color={rgb, 255:red, 208; green, 2; blue, 27 }  ,draw opacity=1 ][line width=0.75]    (184.15,111.79) .. controls (190.55,107.81) and (200.46,105.25) .. (211.57,105.25) .. controls (230.88,105.25) and (246.53,112.96) .. (246.53,122.47) .. controls (246.53,131.98) and (230.88,139.69) .. (211.57,139.69) .. controls (192.27,139.69) and (176.62,131.98) .. (176.62,122.47) .. controls (176.62,120.99) and (176.99,119.55) .. (177.71,118.18) ;  \draw [shift={(181.63,113.58)}, rotate = 333.54] [fill={rgb, 255:red, 208; green, 2; blue, 27 }  ,fill opacity=1 ][line width=0.08]  [draw opacity=0] (10.72,-5.15) -- (0,0) -- (10.72,5.15) -- (7.12,0) -- cycle    ;
			%Curve Lines [id:da5986244313241256] 
			\draw [color={rgb, 255:red, 74; green, 144; blue, 226 }  ,draw opacity=1 ][line width=0.75]    (189.87,122.53) .. controls (123.2,73.07) and (72.53,95.07) .. (67.2,115.07) ;
			\draw [shift={(122.41,93.07)}, rotate = 8.36] [fill={rgb, 255:red, 74; green, 144; blue, 226 }  ,fill opacity=1 ][line width=0.08]  [draw opacity=0] (10.72,-5.15) -- (0,0) -- (10.72,5.15) -- (7.12,0) -- cycle    ;
			%Curve Lines [id:da1311202888462757] 
			\draw [color={rgb, 255:red, 74; green, 144; blue, 226 }  ,draw opacity=1 ] [dash pattern={on 0.84pt off 2.51pt}]  (373.33,157.07) .. controls (400.13,161.47) and (408.53,151.87) .. (421.33,150.67) .. controls (434.13,149.47) and (450.22,154.88) .. (464.53,153.47) .. controls (478.85,152.05) and (490.04,143.27) .. (484.53,133.92) ;
			%Shape: Polygon Curved [id:ds20402312124551036] 
			\draw   (369.03,90.87) .. controls (387.7,76.2) and (443.88,93.13) .. (462.37,92.87) .. controls (480.85,92.61) and (519.97,75.66) .. (538.37,86.87) .. controls (556.77,98.07) and (576.37,126.87) .. (547.7,160.2) .. controls (519.03,193.53) and (492.14,162.74) .. (453.7,161.53) .. controls (415.26,160.33) and (379.7,166.87) .. (365.03,153.53) .. controls (350.37,140.2) and (350.37,105.53) .. (369.03,90.87) -- cycle ;
			%Curve Lines [id:da585382640566275] 
			\draw [line width=0.75]    (383.7,124.43) .. controls (401.14,142.95) and (422.06,142.95) .. (440.37,128.45) ;
			%Curve Lines [id:da45352504680908723] 
			\draw [line width=0.75]    (393.12,128.86) .. controls (401.14,120.4) and (420.32,120.4) .. (430.78,130.87) ;
			%Curve Lines [id:da04319023099270325] 
			\draw [line width=0.75]    (478.53,129.42) .. controls (488.03,140.42) and (510.03,141.42) .. (525.03,130.42) ;
			%Curve Lines [id:da40080822704446284] 
			\draw [line width=0.75]    (482.45,129.52) .. controls (490.47,121.07) and (509.65,121.07) .. (520.11,131.54) ;
			%Shape: Arc [id:dp5519866798332838] 
			\draw  [draw opacity=0][line width=0.75]  (500.38,112.2) .. controls (519.54,112.28) and (535.03,119.97) .. (535.03,129.44) .. controls (535.03,138.97) and (519.38,146.69) .. (500.07,146.69) .. controls (480.77,146.69) and (465.12,138.97) .. (465.12,129.44) .. controls (465.12,119.92) and (480.77,112.2) .. (500.07,112.2) -- (500.07,129.44) -- cycle ; \draw [color={rgb, 255:red, 208; green, 2; blue, 27 }  ,draw opacity=1 ][line width=0.75]    (500.38,112.2) .. controls (519.54,112.28) and (535.03,119.97) .. (535.03,129.44) .. controls (535.03,138.97) and (519.38,146.69) .. (500.07,146.69) .. controls (480.77,146.69) and (465.12,138.97) .. (465.12,129.44) .. controls (465.12,119.74) and (481.35,111.92) .. (501.15,112.21) ; \draw [shift={(498.34,112.22)}, rotate = 360] [fill={rgb, 255:red, 208; green, 2; blue, 27 }  ,fill opacity=1 ][line width=0.08]  [draw opacity=0] (10.72,-5.15) -- (0,0) -- (10.72,5.15) -- (7.12,0) -- cycle    ; 
			%Curve Lines [id:da939575446474578] 
			\draw [color={rgb, 255:red, 74; green, 144; blue, 226 }  ,draw opacity=1 ][line width=0.75]    (464.53,118.92) .. controls (437.63,98.02) and (369.3,98.62) .. (367.47,129.9) ;
			\draw [shift={(406.6,104.84)}, rotate = 357.64] [fill={rgb, 255:red, 74; green, 144; blue, 226 }  ,fill opacity=1 ][line width=0.08]  [draw opacity=0] (10.72,-5.15) -- (0,0) -- (10.72,5.15) -- (7.12,0) -- cycle    ;
			%Curve Lines [id:da5238021300859919] 
			\draw [color={rgb, 255:red, 74; green, 144; blue, 226 }  ,draw opacity=1 ]   (470.03,124.92) .. controls (476.03,129.42) and (478.53,134.42) .. (484.53,133.92) ;
			%Curve Lines [id:da5028889594939359] 
			\draw [color={rgb, 255:red, 74; green, 144; blue, 226 }  ,draw opacity=1 ]   (366.63,138.02) .. controls (366.13,149.9) and (370.13,158.02) .. (373.33,157.07) ;
			%Curve Lines [id:da6482119829242898] 
			\draw [color={rgb, 255:red, 144; green, 19; blue, 254 }  ,draw opacity=1 ]   (535.13,137.57) .. controls (556.8,135.57) and (551.47,99.23) .. (536.8,95.57) .. controls (522.13,91.9) and (506.13,99.23) .. (496.13,100.9) .. controls (486.13,102.57) and (455.44,100.34) .. (443.13,99.57) .. controls (430.83,98.79) and (390.8,90.23) .. (375.8,98.23) .. controls (360.8,106.23) and (348.33,130.84) .. (382.33,143.59) ;
			%Curve Lines [id:da8149604106417865] 
			\draw [color={rgb, 255:red, 144; green, 19; blue, 254 }  ,draw opacity=1 ]   (517.33,135.14) .. controls (522.45,137.13) and (525.33,132.89) .. (529.58,134.89) ;
			%Curve Lines [id:da8549602843956219] 
			\draw [color={rgb, 255:red, 144; green, 19; blue, 254 }  ,draw opacity=1 ]   (383.19,144.16) .. controls (396.05,152.84) and (389.73,158.66) .. (408.33,161.84) ;
			\draw [shift={(384.1,144.79)}, rotate = 212.74] [fill={rgb, 255:red, 144; green, 19; blue, 254 }  ,fill opacity=1 ][line width=0.08]  [draw opacity=0] (10.72,-5.15) -- (0,0) -- (10.72,5.15) -- (7.12,0) -- cycle    ;
			%Curve Lines [id:da20503424382141588] 
			\draw [color={rgb, 255:red, 144; green, 19; blue, 254 }  ,draw opacity=1 ] [dash pattern={on 0.84pt off 2.51pt}]  (408.33,161.84) .. controls (421.83,160.09) and (446.08,153.84) .. (463.08,158.34) .. controls (480.08,162.84) and (499.08,164.59) .. (505.08,158.34) .. controls (511.08,152.09) and (504.58,140.59) .. (517.33,135.14) ;
			
			% Text Node
			\draw (80,154.6) node [anchor=north west][inner sep=0.75pt]  [font=\small,color={rgb, 255:red, 74; green, 144; blue, 226 }  ,opacity=1 ] [align=left] {$\displaystyle b_{1} +a_{1} +b_{2}$};
			% Text Node
			\draw (234.17,134.67) node [anchor=north west][inner sep=0.75pt]  [font=\small,color={rgb, 255:red, 208; green, 2; blue, 27 }  ,opacity=1 ] [align=left] {$\displaystyle a_{2}$};
			% Text Node
			\draw (310.5,163.6) node [anchor=north west][inner sep=0.75pt]  [font=\small,color={rgb, 255:red, 74; green, 144; blue, 226 }  ,opacity=1 ] [align=left] {$\displaystyle b_{1} +a_{1} +b_{2}$};
			% Text Node
			\draw (522.67,141.67) node [anchor=north west][inner sep=0.75pt]  [font=\small,color={rgb, 255:red, 208; green, 2; blue, 27 }  ,opacity=1 ] [align=left] {$\displaystyle a_{2}$};
			% Text Node
			\draw (395.17,62.93) node [anchor=north west][inner sep=0.75pt]  [font=\small,color={rgb, 255:red, 144; green, 19; blue, 254 }  ,opacity=1 ] [align=left] {$\displaystyle b_{1} +a_{1} +b_{2} +a_{2}$};

		\end{tikzpicture}
		
			\caption{$P_{b_1 + a_1 + b_2} \cdot P_{a_2} \cdot [\emptyset]$ (right) and $ P_{a_2} \cdot P_{b_1 + a_1 +b_2 + a_2} \cdot P_{b_1 + a_1 + b_2} \cdot [\emptyset]$ (left) on $\cH_2$}
			\label{subfig: skeins on H_2}
		\end{subfigure}
		\caption{}
		\label{fig: loops}
	\end{figure}

	Thus we only need to prove
	$$
	Q_{b_1 + a_1 + b_2} \cdot Q_{a_2} \cdot [\emptyset] = Q_{a_2} \cdot Q_{b_1 + a_1 +b_2 + a_2} \cdot Q_{b_1 + a_1 + b_2} \cdot [\emptyset],
	$$
	which, after a Dehn twist with respect to $b_1$, is equivalent to :
	$$
	Q_{a_1 + b_2} \cdot Q_{a_2} \cdot [\emptyset] = Q_{a_2} \cdot Q_{a_1 +b_2 + a_2} \cdot Q_{a_1 + b_2} \cdot [\emptyset].
	$$
	
	Take $A =  Q_{a_2}^{-1} \cdot Q_{a_1 + b_2}^{-1} \cdot  Q_{a_2} \cdot Q_{a_1 +b_2 + a_2} \cdot Q_{a_1 + b_2} \in \widehat{\Sk}(\Sigma_2)$. We will show that 
	\begin{align}
		\Ad_A (P_{-b_1} ) &= P_{-b_1}, \quad \text{and} \label{eq: mutation 1}\\
		\Ad_A (P_{b_3}) &= P_{b_3}. \label{eq: mutation 3}
	\end{align}
	Recall that the minus of a loop refers to the same loop but with the opposite direction.
	We use Proposition \ref{prop: Ad action} to do the computations:
	\begin{align*}
		P_{-b_1}  &\xmapsto{\Ad_{Q_{a_1 + b_2}} } P_{-b_1} + P_{-b_1 + a_1 + b_2}\\
		&\xmapsto{\Ad_{Q_{a_1 + b_2 + a_2}}} P_{-b_1 } + P_{-b_1 + a_1 + b_2 + a_2} + P_{-b_1 + a_1 + b_2}\\
		&\xmapsto{\Ad_{Q_{a_2}}}   P_{-b_1} + P_{-b_1 + a_1 + b_2} \\
		&\xmapsto{\Ad_{Q_{a_1 + b_2}^{-1}}} P_{-b_1}\\
		&\xmapsto{\Ad_{Q_{a_2}^{-1}} } P_{- b_1}
	\end{align*}
	The second line is because that we can isotope the loops $-b_1 + a_1 + b_2$ and ${a_1 + b_2 + a_2}$ so that they do not intersect (Figure \ref{fig: nonintersecting} left). Thus $\Ad_{Q_{a_1 + b_2 + a_2}} P_{-b_1 + a_1 + b_2}= P_{-b_1 + a_1 + b_2}$.
	
	\begin{figure}[htbp]

	\tikzset{every picture/.style={line width=0.75pt}} %set default line width to 0.75pt        
	
	\begin{tikzpicture}[x=0.75pt,y=0.75pt,yscale=-1,xscale=1]
		
		%Shape: Polygon Curved [id:ds4876879712384803] 
		\draw   (156.53,69.6) .. controls (181.2,62.27) and (170,77.33) .. (231.87,75.6) .. controls (293.73,73.87) and (317.87,65.6) .. (333.87,82.93) .. controls (349.87,100.27) and (327.2,136.93) .. (303.2,143.6) .. controls (279.2,150.27) and (253.89,139.36) .. (225.87,140.93) .. controls (197.84,142.51) and (161.2,156.27) .. (141.2,133.6) .. controls (121.2,110.93) and (131.87,76.93) .. (156.53,69.6) -- cycle ;
		%Curve Lines [id:da4103251862522159] 
		\draw    (167.87,108.27) .. controls (173.87,116.93) and (187.87,122.27) .. (206.53,110.93) ;
		%Curve Lines [id:da5528909594150124] 
		\draw    (175.2,109.6) .. controls (182.53,100.93) and (192.53,102.27) .. (197.87,111.6) ;
		%Curve Lines [id:da23697100638845914] 
		\draw    (247.07,106.93) .. controls (252.23,114.39) and (267.73,117.6) .. (282.4,107.6) ;
		%Curve Lines [id:da0480321220355584] 
		\draw    (251.87,105.6) .. controls (259.2,96.93) and (269.2,98.27) .. (274.53,107.6) ;
		%Curve Lines [id:da33936919148024103] 
		\draw [color={rgb, 255:red, 208; green, 2; blue, 27 }  ,draw opacity=1 ]   (197.2,116.27) .. controls (189.2,125.6) and (173.2,129.6) .. (165.2,120.27) .. controls (157.2,110.93) and (155.44,94.12) .. (173.2,88.27) .. controls (190.96,82.41) and (259.87,85.6) .. (275.87,90.93) .. controls (291.87,96.27) and (299.55,102.13) .. (300.53,110.27) .. controls (301.52,118.4) and (285.33,131.6) .. (273.87,129.6) .. controls (262.4,127.6) and (262.4,115.6) .. (250.53,110.93) ;
		%Curve Lines [id:da08499691495896577] 
		\draw [color={rgb, 255:red, 208; green, 2; blue, 27 }  ,draw opacity=1 ] [dash pattern={on 0.84pt off 2.51pt}]  (197.2,116.27) .. controls (200.3,113.94) and (207.2,110.27) .. (213.87,112.93) .. controls (220.53,115.6) and (227.2,116.93) .. (233.87,114.27) .. controls (240.53,111.6) and (246.06,107.05) .. (250.53,110.93) ;
		%Straight Lines [id:da44147284029464995] 
		\draw [color={rgb, 255:red, 208; green, 2; blue, 27 }  ,draw opacity=1 ]   (219.2,85.6) -- (226.53,82.27) ;
		%Straight Lines [id:da036778586640586086] 
		\draw [color={rgb, 255:red, 208; green, 2; blue, 27 }  ,draw opacity=1 ]   (226.53,89.6) -- (219.2,85.6) ;
		%Curve Lines [id:da5852167972052515] 
		\draw [color={rgb, 255:red, 74; green, 144; blue, 226 }  ,draw opacity=1 ]   (257.2,100.27) .. controls (255.41,93.47) and (244.53,93.6) .. (239.2,93.6) .. controls (233.87,93.6) and (225.87,94.27) .. (219.2,91.6) .. controls (212.53,88.93) and (203.87,89.6) .. (196.53,92.27) .. controls (189.97,94.65) and (186.51,98.45) .. (182.6,102.26) ;
		\draw [shift={(181.2,103.6)}, rotate = 316.97] [color={rgb, 255:red, 74; green, 144; blue, 226 }  ,draw opacity=1 ][line width=0.75]    (10.93,-3.29) .. controls (6.95,-1.4) and (3.31,-0.3) .. (0,0) .. controls (3.31,0.3) and (6.95,1.4) .. (10.93,3.29)   ;
		%Curve Lines [id:da8257331109530315] 
		\draw [color={rgb, 255:red, 74; green, 144; blue, 226 }  ,draw opacity=1 ] [dash pattern={on 0.84pt off 2.51pt}]  (181.2,103.6) .. controls (173.87,113.6) and (167.2,119.6) .. (155.87,120.93) .. controls (144.53,122.27) and (137.87,122.67) .. (131.87,114.93) ;
		%Curve Lines [id:da5152114239573746] 
		\draw [color={rgb, 255:red, 74; green, 144; blue, 226 }  ,draw opacity=1 ]   (131.87,114.93) .. controls (130.17,106.21) and (137.22,102.37) .. (143.2,104.67) .. controls (149.18,106.96) and (155.78,124.92) .. (167.2,132.93) .. controls (178.62,140.95) and (191.56,137.49) .. (205.2,131.6) .. controls (218.84,125.71) and (226.53,119.6) .. (225.2,108.27) .. controls (223.87,96.93) and (194.53,93.6) .. (193.87,105.6) ;
		%Curve Lines [id:da6671770933986545] 
		\draw [color={rgb, 255:red, 74; green, 144; blue, 226 }  ,draw opacity=1 ] [dash pattern={on 0.84pt off 2.51pt}]  (193.87,105.6) .. controls (195.63,110.65) and (209.7,109.27) .. (215.8,108.05) .. controls (221.9,106.83) and (226.1,103.15) .. (233.2,101.6) .. controls (240.3,100.05) and (248.8,107.05) .. (257.2,100.27) ;
		%Shape: Polygon Curved [id:ds632980192307965] 
		\draw   (428.53,70.93) .. controls (453.2,63.6) and (442,78.67) .. (503.87,76.93) .. controls (565.73,75.2) and (589.87,66.93) .. (605.87,84.27) .. controls (621.87,101.6) and (599.2,138.27) .. (575.2,144.93) .. controls (551.2,151.6) and (525.89,140.69) .. (497.87,142.27) .. controls (469.84,143.84) and (433.2,157.6) .. (413.2,134.93) .. controls (393.2,112.27) and (403.87,78.27) .. (428.53,70.93) -- cycle ;
		%Curve Lines [id:da9995405002262685] 
		\draw    (439.87,109.6) .. controls (445.87,118.27) and (459.87,123.6) .. (478.53,112.27) ;
		%Curve Lines [id:da16457777833231235] 
		\draw    (447.2,110.93) .. controls (454.53,102.27) and (464.53,103.6) .. (469.87,112.93) ;
		%Curve Lines [id:da9313944856104588] 
		\draw    (516.37,109.1) .. controls (522.37,117.77) and (536.37,123.1) .. (555.03,111.77) ;
		%Curve Lines [id:da4495776553412576] 
		\draw    (523.03,110.43) .. controls (530.37,101.77) and (540.37,103.1) .. (545.7,112.43) ;
		%Curve Lines [id:da5648863380273612] 
		\draw [color={rgb, 255:red, 208; green, 2; blue, 27 }  ,draw opacity=1 ]   (459.87,118.27) .. controls (447.7,120.05) and (443.7,133.55) .. (429.2,126.05) .. controls (414.7,118.55) and (418.1,96.12) .. (435.87,90.27) .. controls (453.63,84.41) and (532.53,87.6) .. (551.2,89.6) .. controls (569.87,91.6) and (583.03,88.88) .. (593.7,93.55) .. controls (604.37,98.22) and (609.2,106.93) .. (607.2,116.27) ;
		%Curve Lines [id:da37082990882690514] 
		\draw [color={rgb, 255:red, 208; green, 2; blue, 27 }  ,draw opacity=1 ] [dash pattern={on 0.84pt off 2.51pt}]  (459.87,118.27) .. controls (471.2,121.6) and (466.7,130.1) .. (482.7,133.55) .. controls (498.7,137) and (516.37,130.38) .. (529.7,131.05) .. controls (543.03,131.72) and (565.87,136.55) .. (577.2,134.55) .. controls (588.53,132.55) and (599.2,126.27) .. (607.2,116.27) ;
		%Straight Lines [id:da3913864690172282] 
		\draw [color={rgb, 255:red, 208; green, 2; blue, 27 }  ,draw opacity=1 ]   (491.2,86.93) -- (498.53,83.6) ;
		%Straight Lines [id:da5219469961032113] 
		\draw [color={rgb, 255:red, 208; green, 2; blue, 27 }  ,draw opacity=1 ]   (498.53,90.93) -- (491.2,86.93) ;
		%Shape: Ellipse [id:dp8580193555590769] 
		\draw  [color={rgb, 255:red, 74; green, 144; blue, 226 }  ,draw opacity=1 ] (500.37,111.13) .. controls (500.37,102.68) and (515.62,95.83) .. (534.43,95.83) .. controls (553.25,95.83) and (568.5,102.68) .. (568.5,111.13) .. controls (568.5,119.58) and (553.25,126.43) .. (534.43,126.43) .. controls (515.62,126.43) and (500.37,119.58) .. (500.37,111.13) -- cycle ;
		%Straight Lines [id:da6232754880549776] 
		\draw [color={rgb, 255:red, 74; green, 144; blue, 226 }  ,draw opacity=1 ]   (503.7,104.43) -- (507.7,97.77) ;
		%Straight Lines [id:da7987396933397104] 
		\draw [color={rgb, 255:red, 74; green, 144; blue, 226 }  ,draw opacity=1 ]   (503.7,104.43) -- (511.7,104.43) ;
		
		% Text Node
		\draw (130.67,145.33) node [anchor=north west][inner sep=0.75pt]  [font=\small,color={rgb, 255:red, 74; green, 144; blue, 226 }  ,opacity=1 ] [align=left] {$\displaystyle -b_{1} +a_{1} +b_{2}$};
		% Text Node
		\draw (252,148) node [anchor=north west][inner sep=0.75pt]  [font=\small,color={rgb, 255:red, 208; green, 2; blue, 27 }  ,opacity=1 ] [align=left] {$\displaystyle a_{1} +b_{2} +a_{2}$};
		% Text Node
		\draw (406.8,148.67) node [anchor=north west][inner sep=0.75pt]  [font=\small,color={rgb, 255:red, 208; green, 2; blue, 27 }  ,opacity=1 ] [align=left] {$\displaystyle a_{1} +b_{2} +a_{2} +b_{3}$};
		% Text Node
		\draw (572.8,105.67) node [anchor=north west][inner sep=0.75pt]  [font=\small,color={rgb, 255:red, 74; green, 144; blue, 226 }  ,opacity=1 ] [align=left] {$\displaystyle a_{2}$};

	\end{tikzpicture}
	
		\caption{Some loops on $\Sigma_2$.}
		\label{fig: nonintersecting}		
	\end{figure}
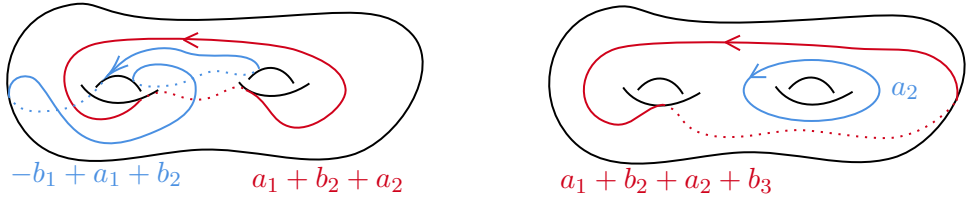

	Similarly, we have:
	\begin{align*}
		P_{b_3} 
		&\xmapsto{\Ad_{Q_{a_1 + b_2}} } P_{b_3}\\
		&\xmapsto{\Ad_{Q_{a_1 + b_2 + a_2}}} P_{b_3 } + P_{a_1 + b_2 + a_2 + b_3}\\
		&\xmapsto{\Ad_{Q_{a_2}}} P_{b_3} + P_{a_2 + b_3}   + P_{a_1 + b_2 + a_2 + b_3} \\
		&\xmapsto{\Ad_{Q_{a_1 + b_2}^{-1}}} P_{b_3} + P_{a_2 + b_3}\\
		&\xmapsto{\Ad_{Q_{a_2}^{-1}} } P_{b_3}
	\end{align*}
	where the third line is because that we can isotope $a_2$ and $a_1 + b_2 + a_2 + b_3$ so that they do not intersect (Figure \ref{fig: nonintersecting} right). Hence \eqref{eq: mutation 1} and \eqref{eq: mutation 3} are proved.

    Since \(b_1\) bounds a disk, combining \eqref{eq: mutation 1} we know that
    \[
     \Ad_A (P_{-b_1} - \bigcirc)\cdot [\emptyset]  = (P_{-b_1} - \bigcirc ) \cdot [\emptyset] = 0,
    \]
    multiplying \(A^{-1}\) on both sides will result in:
    \[
		(P_{-b_{1}} - \bigcirc) \cdot A^{-1} \cdot [\emptyset] = 0.
    \]

    Similarly by \eqref{eq: mutation 3} we have : 
    \[
     (P_{b_3} - \bigcirc) \cdot A^{-1} \cdot [\emptyset] = 0. 
    \]

	By Lemma \ref{lem: sliding}, we have that $A^{-1} \cdot [\emptyset]$ lies in the image in $\Sk(\cH_2)$ of the skein module of $\cH_2$ minus the two disks with boundary $b_1$ and $b_3$.  Since this latter space is contractible, $A^{-1} \cdot [\emptyset]$ must be a scalar. Checking the constant term forces $A^{-1} \cdot [\emptyset]$ to be $1 \cdot [\emptyset]$.

\end{proof}

\begin{spacing}{1}

	\bibliographystyle{alpha}
	\bibliography{refs.bib}

\end{spacing}
	
\end{document}